\documentclass[11pt]{amsart}

\usepackage{latexsym,graphicx,amssymb,mathabx}
\usepackage[tips,matrix,arrow]{xy}
\usepackage{etex}
\usepackage{tikz}
\usepackage{diagbox}
\usepackage{adjustbox}

\usepackage[T1]{fontenc}

\usepackage{xcolor}
\usepackage{ulem}

\definecolor{violet}{cmyk}{1,0,0,0}

\usepackage{hyperref}

\usepackage{setspace}

\numberwithin{equation}{subsection}

\newcommand{\C}{\mathbb{C}}

\newcommand{\R}{\mathbb{R}}
\newcommand{\Z}{\mathbb{Z}}
\renewcommand{\H}{\mathbb{H}}

\newcommand{\fsl}{\mathfrak{sl}}

\newcommand{\Ker}{\mathrm{Ker}}

\newcommand{\Ind}{\mathrm{Ind}}

\newcommand{\mult}{\mathrm{mult}}

\newcommand{\pr}{\mathrm{pr}}

\newcommand{\norm}[1]{\left\vert\!\left\vert {#1} \right\vert\!\right\vert}
\newcommand{\modo}{\;\mathrm{mod}\;}

\newtheorem{thm}{Theorem}[section]
\newtheorem{prop}[thm]{Proposition}
\newtheorem{lemma}[thm]{Lemma}
\newtheorem{defn}{Definition}[section]

\newtheorem{rem}[thm]{Remark}

\newcommand{\ch}{\operatorname{ch}}
\newcommand{\sch}{\operatorname{sch}}

\newcommand{\sdim}{\operatorname{sdim}}

\newcommand{\fg}{\mathfrak{g}}
\newcommand{\fh}{\mathfrak{h}}

\newcommand{\fb}{\mathfrak{b}}

\newcommand{\fn}{\mathfrak{n}}

\newcommand{\vep}{\varepsilon}

\newcommand{\cE}{\mathcal{E}}

\newcommand{\cO}{\mathcal{O}}
\newcommand{\cP}{\mathcal{P}}

\definecolor{mygreen}{cmyk}{1,0,1,0}

\newcommand{\node}[1]{\begin{tikzpicture} \draw [fill=gray!#1] (0,0) circle(0.1); \end{tikzpicture}}

\begin{document}
\title{Hidden $\mathrm{SL}_2(\Z)$-symmetry in $BC_l^{(2)}$}
\author{K. Iohara and Y. Saito}
\address{Universit\'{e} Claude Bernard Lyon 1, CNRS, Institut Camille Jordan UMR 5208, F-69622 Villeurbanne, France}
\email{iohara@math.univ-lyon1.fr}
\address{Department of Mathematics, Rikkyo University, Toshima-ku, Tokyo 171-8501, Japan}
\email{yoshihisa@rikkyo.ac.jp}

\thanks{Y .S. is partially supported by JSPS KAKENHI Grant Number 20K03568.}

\begin{abstract} A well-known $\Gamma_\theta$-action on the characters of integrable highest weight modules over the affine Lie algebra of type $BC_l^{(2)}$ at a positive level
is extended to an $\mathrm{SL}_2(\Z)$-action at a positive even level by supplementing their twisted characters. 
\end{abstract}


\maketitle

\setcounter{tocdepth}{1}
\tableofcontents

\section*{Introduction}
In 1972, I. G. Macdonald \cite{Macdonald1972} classified the so-called (not necessarily reduced) affine root systems and obtained some identities on the Fourier expansions of products of the Dedekind $\eta$-function associated with each affine root systems. The identities he obtained are named after him. Independently, F. Bruhat and J. Tits \cite{BruhatTits1972} classified affine root systems, where some of their results are explained in a survey article \cite{Tits1979}.  

To obtain the Macdonald identities, he started from an identity, that looks like a denominator identity up to some factor, called ``missing factor''. Later, it was shown by several authors (see, e.g., our expository note \cite{IoharaSaito202} for the references), that the ``missing factor'' corresponds to the imaginary roots and the identity Macdonald used is the denominator identity of an affine Lie algebra. It would be a natural question to see what are the nature of the characters of integrable highest weight modules over an affine Lie algebra. Indeed, E. Looijenga \cite{Looijenga1976/77} interpreted 
a character of an affine Lie algebra of untwisted type as a section of a line bundle over an abelian variety, which is a product of an elliptic curve. This fact inspired V. G. Kac and D. H. Peterson, indeed, in 1984, they \cite{KacPeterson1984} had studied ``modular properties'' of the characters of affine Lie algebras. 
They showed the following:
\begin{enumerate}
\item for an untwisted affine Lie algebra, the group $\mathrm{SL}_2(\Z)$ acts on the $\C$-span of the normalized characters of each fixed central charge and 
\item for a twisted affine Lie algebra, a congruence subgroup of $\mathrm{SL}_2(\Z)$ acts on the $\C$-span of the of the normalized characters of each fixed central charge
\end{enumerate}
via ``modular transformation''. For $BC_l^{(2)}$ (or $A_{2l}^{(2)}$), this congruence subgroup is the so-called $\theta$-subgroup $\Gamma_\theta$. Here, the characters have been regarded as functions on a certain finite dimensional complex subdomain $Y$ of the affine Lie algebra of type $BC_l^{(2)}$. 

In Macdonald's paper \cite{Macdonald1972}, he considered several possible evaluations of the denominator identities. Indeed, for $BC_l^{(2)}$, in its Appendix 1,  Type $BC_l$, (6), he treated $4$ different evaluations which concluded with $4$ different $\eta$-products. The case treated by V. Kac and D. Peterson corresponds to the case (b) (cf. Remark \ref{rem:Macdonald-id}). By computing the modular transformations of these $4$ $\eta$-products, we observed that there seems to be a bigger hidden symmetry in $BC_l^{(2)}$, which inspired us to work on the subject of this article. 

In this article, we show that the group $\mathrm{SL}_2(\Z)$ acts on the vector space spanned by the characters and twisted-characters of integrable highest weight modules over the affine Lie algebra $\fg$ of type $BC_l^{(2)}$, for each fixed \textit{even} level.  Here, we should consider the characters and twisted-characters as functions on a complex subdomain of the dual of the Cartan subalgebra of $\fg$ and introduce \textit{two} coordinate systems dependent on the choice of special indices. 
Moreover, we observe that the characters and twisted-characters of the affine Lie algebra of type $BC_l^{(2)}$ at \textit{even} level can be identified with the characters and super-characters of the affine Lie superalgebra of type $B^{(1)}(0,l)$  ! 

This article is organized as follows. In Section $1$, we recall some basic notions and setting on $BC_l^{(2)}$, in particular, depending on the choice of a special index (see Definition \ref{defn:special-index}), we present two coordinate systems of a complex domain in the dual of the Cartan subalgebra of the affine Lie algebra of type $BC_l^{(2)}$. 
In Section $2$, we recall the definition of the (twisted-)character of an object of the category $\cO$ of Bernstein-Gelfand-Gelfand. In Sections 3, we introduce theta-series and twisted invariants of the affine Weyl group of type $BC_l^{(2)}$ and express the normalized (twisted-)characters in terms these (twisted-)invariants. In Section $4$, we compute the modular transformations of (twisted-)characters. Finally, in Section 5, we interpret the computational results obtained in the previous sections. We conclude with an appendix explaining some relevant basic facts about Lie superalgebras, for the sake of reader's convenience.

\section{Kac-Moody Lie algebras of type $BC_l^{(2)}$}\label{sect:BC_l}
In this article, we deal with the affine root system of type $BC_l^{(2)}\ (l\geq 1)$
in the sense of K. Saito \cite{Saito1985}. Since 
there are many nomenclatures for affine root systems, we give   
the correspondence between other nomenclatures.
\vskip 3mm
\begin{center}
\begin{tabular}{|c||c|c|c|c|c||c|c|c|} \hline
Saito $($\cite{Saito1985}$)$ &
Kac $($\cite{Kac1969}$)$ &  Moody $($\cite{Moody1969}$)$
& Macdonald $($\cite{Macdonald1972}$)$ & 
Carter $($\cite{Carter2005}$)$\\ \hline
$BC_l^{(2)}$ & $A_{2l}^{(2)}$ &$\begin{matrix}
{\tiny A_{1,2}} & \text{\tiny $l=1$}\\ 
{\tiny BC_{l,2}} & \text{\tiny $l>1$} 
\end{matrix}$ & $BC_l=BC_l^\vee$ & $\widetilde{C}_l'$ \\ \hline
\end{tabular}
\end{center} 
\vskip 3mm
\subsection{Root datum}\label{subsect:root-data}
For a positive integer $l$, let $A=(a_{i,j})$ be an 
$(l+1)\times(l+1)$ generalized Cartan matrices (GCM for short) of type $BC_l^{(2)}$:
\[A=\begin{pmatrix}
2 & -1\\
-4 & 2
\end{pmatrix}\quad (l=1),\quad A=
{\small
\begin{pmatrix}
2  & -1 & &&\\
-2 & 2 & -1 &&\\
& -1 & 2  & -1 &&\\
&&\ddots & \ddots & \ddots \\
&&& -1 & 2  & -1 \\
&&&  & -2  & 2 
\end{pmatrix}}\quad (l\geq 2).
\]
In the following, we construct a realization (in the sense of Kac \cite{Kac1990}) of 
$A$ (see \eqref{defn:Cartan} below).\\

Let $F_f$ be an $l$-dimensional real vector space equipped with a positive definite 
symmetric bilinear form $I_f:F_f\times F_f\longrightarrow \R$. 
Fix an orthonormal basis $\vep_1,\ldots,\vep_l$ of $F_f$ with respect to $I_f$.
Consider a $2$-dimensional extension 
$\widehat{F}:=F_f\bigoplus \R\delta \bigoplus \R \gamma$ equipped with a 
non-degenerate symmetric bilinear form $\widehat{I}$ on $\widehat{F}$ defined by
\[\widehat{I}\big|_{F_f\times F_f}=I_f,\quad
\widehat{I}(\vep_i,\delta)=\widehat{I}(\vep_i,\gamma)=0 \ (1\leq i\leq l),\]
\[\widehat{I}(\delta,\delta)=\widehat{I}(\gamma,\gamma)=0,\quad
\widehat{I}(\delta,\gamma)=1.\]
Since $\widehat{I}$ is non-degenerate, there is an induced isomorphism 
$\kappa:\widehat{F}\xrightarrow{\sim}\widehat{F}^\ast:=\mathrm{Hom}_\R(\widehat{F},\R)$ 
defined by
\begin{equation*}
\langle \kappa(u),v\rangle=\widehat{I}(u,v)\quad \text{for }u,v\in\widehat{F},
\end{equation*}
where $\langle \cdot,\cdot\rangle:\widehat{F}^\ast \times \widehat{F}\longrightarrow \R$
is a canonical pairing. 
Set 
\begin{equation}\label{defn:Cartan}
\mathfrak{h}=\widehat{F}^\ast\otimes\C.
\end{equation}
Then, its complex dual space $\mathfrak{h}^\ast:=\mathrm{Hom}_\C(\mathfrak{h},\C)$ 
is naturally identified with $\widehat{F}\otimes \C$. The complexification of the 
isomorphism $\kappa:\widehat{F}\xrightarrow{\sim}\widehat{F}^\ast$ is denoted by 
the same symbol $\kappa:\fh^\ast\overset{\sim}{\longrightarrow}\fh$. 

\begin{rem}
In Kac's book \cite{Kac1990}, our 
$\kappa^{-1}:\fh\overset{\sim}{\longrightarrow}\fh^\ast$ is denoted by $\nu$.
\end{rem}

Set 
\[\alpha_0=\delta-2\vep_1,\quad 
\alpha_i=\vep_i-\vep_{i+1}\ \ (1\leq i<l),\quad
\alpha_l=\vep_l,\]
and define a linearly independent subset $\Pi$ of $\fh^\ast$ (called 
the set of simple roots) by
\[\Pi=\{\alpha_0,\ldots, \alpha_l\}.\]
For a non-isotropic vector $\alpha\in\widehat{F}$, denote 
$\alpha^\vee=2\alpha/\widehat{I}(\alpha,\alpha)$. Set 
\[\Pi^\vee=\{h_0\ldots, h_l\}
\subset \fh
\quad \text{where }h_i:=
\kappa\big(\alpha_i^\vee \big)\ \ (0\leq i\leq l).\]
It is called the set of simple coroots. Hence, the triple 
$\big(\fh,\Pi,\Pi^\vee\big)$ is a 
realization of a GCM $A$. Its Dynkin
diagram is given as follows.
\begin{center}
\begin{tikzpicture}
\draw (0,0) circle (0.1); \draw (0,-0.1) node[below] {$\alpha_0$};
\draw (0.1, 0) -- (0.9,0); 
\draw (0.45, 0.1) -- (0.55,0); 
\draw (0.45, -0.1) -- (0.55,0); 
\draw (0.5, 0.1) node[above] {$4$}; 
\draw (1,0) circle (0.1);  \draw (1,-0.1) node[below] {$\alpha_1$};
\draw (1.5,0) node[right] {$(l=1)$,};
\end{tikzpicture}
\quad
\begin{tikzpicture}
\draw (0,0) circle (0.1); \draw (0,-0.1) node[below] {$\alpha_0$};
\draw (0.45, 0.1) -- (0.55,0); 
\draw (0.45, -0.1) -- (0.55,0); 
\draw (0.5,0.1)  node[above] {$2$};
\draw (0.1,0) -- (0.9,0) ;
\draw (1,0) circle (0.1);  \draw (1,-0.1) node[below] {$\alpha_1$};
\draw (1.1,0) -- (1.9,0); 
\draw (2,0) circle (0.1); 
\draw [dashed] (2.1,0) -- (3.9,0);
\draw (4,0) circle (0.1); 
\draw (4.1,0) -- (4.9,0); 
\draw (5,0) circle (0.1); \draw (5,-0.1) node[below] {$\alpha_{l-1}$};
\draw (5.1,0) -- (5.9,0);
\draw (6,0) circle (0.1);  \draw (6,-0.1) node[below] {$\alpha_{l}$};
\draw (5.45,0.1) -- (5.55,0) ;
\draw (5.45,-0.1) -- (5.55,0) ;
\draw (5.5,0.1) node[above] {$2$};
\draw (6.5,0) node[right] {$(l\geq 2)$,};
\end{tikzpicture}
\end{center}

The labels $a_i\ (0\leq i\leq l)$ and co-labels $a_i^\vee\ (0\leq i\leq l)$ of $A$
are by definition, relatively prime positive integers satisfying
\[\begin{aligned}
&A{}^t\mathbf{a}={}^t\mathbf{0}\quad \text{where }
\mathbf{a}:= (a_0,\ldots,a_l),\ 
\mathbf{0}:=(0,\ldots,0)\quad\text{and}\\
&\mathbf{a}^\vee A=\mathbf{0}\quad \text{where }
\mathbf{a}^\vee:= (a_0^\vee,\ldots,a_l^\vee),
\end{aligned}\]
respectively. Explicit forms of them are given by
\[a_i=\begin{cases}
1 & (i=0),\\
2 & (1\leq i\leq l),
\end{cases}\quad 
a_i^\vee=\begin{cases}
2 & (0\leq i\leq l-1),\\
1 & (i=l).
\end{cases}\]
It is known that the following formula holds:
\begin{equation}\label{eq:delta}
\delta=\sum_{i=0}^l a_i\alpha_i.
\end{equation}

\subsection{Kac-Moody Lie algebras and their roots}
\subsubsection{}
Let $\fg=\fg(A)$ be the Kac-Moody Lie algebra associated with the generalized Cartan 
matrix $A$ of type $BC_l^{(2)}$. It is a Lie algebra
over $\C$ generated by $e_i,f_i\ (0\leq i\leq l),h\in \mathfrak{h}$ 
subject to the following relations:
\begin{itemize}
\item[(L1)] $[h,h']=0$ for $h,h'\in \mathfrak{h}$,
\vskip 1mm
\item[(L2)] $\left\{\begin{array}{l}
[h,e_i]=\langle h,\alpha_i\rangle \, e_i,\\[3pt] 
\lbrack h,f_i\rbrack=-\langle h,\alpha_i\rangle \, f_i,
\end{array}\right.$ for
$h\in\mathfrak{h}$ and $0\leq i\leq l$,
\vskip 1mm
\item[(L3)] $[e_i,f_j]=\delta_{i,j}h_i$ for $0\leq i,j\leq l$,
\vskip 1mm
\item[(L4)] $\left\{\begin{array}{l}
\mathrm{ad}(e_i)^{1-a_{i,j}}(e_j)=0,\\[3pt] 
\mathrm{ad}(f_i)^{1-a_{i,j}}(f_j)=0,
\end{array}\right.$ for $0\leq i\ne j\leq l$.
\end{itemize} 
The abelian subalgebra $\fh$ of $\fg$ is called the Cartan subalgebra of $\fg$.\\

In the rest of this article, we use standard notations in the theory of Kac-Moody Lie
algebras following Kac's book \cite{Kac1990}.  
Namely, let $\fn_{\pm}$ be 
the subalgebra of $\fg$ generated by $\{e_i\}_{0\leq i\leq l}$ 
(resp. $\{f_i\}_{0\leq i\leq l}$). We have the triangular decomposition of $\fg$:
\begin{equation}\label{eq:triangular-decomp}
\fg=\fn_+\bigoplus \fh\bigoplus \fn_-.
\end{equation}
Let $\Delta$ be the set of roots, $\Delta_{\pm}$ the set of positive (resp. negative) 
roots and $Q$ the root lattice. The Weyl group
$W$ is the subgroup of $\mathrm{O}(\fh^\ast,\widehat{I})$ generated by
simple reflections $s_\alpha\ (\alpha\in \Pi)$. 

Let $\Delta^{re}=W.\Pi$ (resp. $\Delta^{im}=\Delta\setminus \Delta^{re}$) be the set of 
real (resp. imaginary) roots. Their explicit forms are given by
\begin{equation}\label{eq:real-roots}
\Delta^{re}=\Delta_s^{re}\amalg \Delta_m^{re} \amalg \Delta_l^{re}\quad \text{and}\quad
\Delta^{im}=\Z\delta
\end{equation}
where
\begin{align*}
\Delta_s^{re}&=\{\pm \vep_i+r\delta\,|\,1\leq i\leq l,\, r\in\Z\},\\
\Delta_m^{re}&=\{\pm \vep_i\pm \vep_j+r\delta\,|\,1\leq i<j\leq l,\, r\in\Z\},\\
\Delta_l^{re}&=\{\pm2 \vep_i+(2r+1)\delta\,|\,1\leq i\leq l,\, r\in\Z\}.
\end{align*}

%
\subsubsection{}
\begin{defn}\label{defn:special-index}
An index $0\leq i_0\leq l$ of the set $\Pi$ of simple roots is called \textbf{special} 
if  there exist $p\in\Z_{>0}$ and 
$\alpha\in \Delta_+$ such that $\delta-a_{i_0}\alpha_{i_0}=p\alpha$. 
\end{defn}

There are two special indices of $\Pi$. 
\begin{itemize}
\item[(i)] Since $\delta-a_0\alpha_0=2\vep_1$ and $\vep_1\in\Delta$, 
the index $0$ is special. 
\vskip 1mm
\item[(ii)] 
Since $\delta-a_l\alpha_l=\delta-2\vep_l\in\Delta$, the index $l$ is special. 
\end{itemize}
\vskip 1mm
\noindent
{\bf (i)} We call the index $i_0=0$ the special index of type 
$\mathrm{I}$. For later convenience,  we denote
$\vep_i^{(\mathrm{I})}=\vep_i\ (1\leq i\leq l)$, 
$\alpha_i^{(\mathrm{I})}=\alpha_i\ (0\leq i\leq l)$ and 
$F_f^{(\mathrm{I})}=F_f$. 
It follows that  
\begin{equation}
F_f^{(\mathrm{I})}=\bigoplus_{i=1}^l \R\vep_i^{(\mathrm{I})}
=\bigoplus_{i=1}^l \R\alpha_i^{(\mathrm{I})}\quad\text{and}\quad
F:=\bigoplus_{i=0}^l\R \alpha_i^{(\mathrm{I})}=F_f^{(\mathrm{I})}\bigoplus \R\delta.
\end{equation}
\vskip 1mm
\noindent
{\bf (ii)} We call the index $i_0=l$ the special index of 
type $\mathrm{I\hspace{-1.2pt}I}$. Replace the numeration of 
simple roots with the following:
\[\alpha_i^{(\mathrm{I\hspace{-1.2pt}I})}:=\alpha_{l-i}\quad (0\leq i\leq l). \]
In this new numeration, the index $0$ is special and the corresponding 
Dynkin diagram is given by 
\begin{center}
\begin{tikzpicture}
\draw (0,0) circle (0.1); \draw (0,-0.1) node[below] 
{$\alpha_0^{(\mathrm{I}\hspace{-1.2pt}\mathrm{I})}$};
\draw (0.1, 0) -- (0.9,0); 
\draw (0.55, 0.1) -- (0.45,0); 
\draw (0.55, -0.1) -- (0.45,0); 
\draw (0.5, 0.1) node[above] {$4$}; 
\draw (1,0) circle (0.1);  \draw (1,-0.1) node[below] 
{$\alpha_1^{(\mathrm{I}\hspace{-1.2pt}\mathrm{I})}$};
\draw (1.5,0) node[right] {$(l=1)$,};
\end{tikzpicture}
\quad
\begin{tikzpicture}
\draw (0,0) circle (0.1); \draw (0,-0.1) node[below] 
{$\alpha_0^{(\mathrm{I}\hspace{-1.2pt}\mathrm{I})}$};
\draw (0.55, 0.1) -- (0.45,0); 
\draw (0.55, -0.1) -- (0.45,0); 
\draw (0.5,0.1)  node[above] {$2$};
\draw (0.1,0) -- (0.9,0) ;
\draw (1,0) circle (0.1);  \draw (1,-0.1) node[below] 
{$\alpha_1^{(\mathrm{I}\hspace{-1.2pt}\mathrm{I})}$};
\draw (1.1,0) -- (1.9,0); 
\draw (2,0) circle (0.1); 
\draw [dashed] (2.1,0) -- (3.9,0);
\draw (4,0) circle (0.1); 
\draw (4.1,0) -- (4.9,0); 
\draw (5,0) circle (0.1); \draw (5,-0.1) node[below] 
{$\alpha_{l-1}^{(\mathrm{I}\hspace{-1.2pt}\mathrm{I})}$};
\draw (5.1,0) -- (5.9,0);
\draw (6,0) circle (0.1);  \draw (6,-0.1) node[below] 
{$\alpha_{l}^{(\mathrm{I}\hspace{-1.2pt}\mathrm{I})}$};
\draw (5.55,0.1) -- (5.45,0) ;
\draw (5.55,-0.1) -- (5.45,0) ;
\draw (5.5,0.1) node[above] {$2$};
\draw (6.5,0) node[right] {$(l\geq 2)$.};
\end{tikzpicture}
\end{center}
Set 
\[\vep_i^{(\mathrm{I\hspace{-1.2pt}I})}=-\vep_{l+1-i}+\dfrac12\delta\quad (1\leq i\leq l)
\quad\text{and}\quad F_f^{(\mathrm{I\hspace{-1.2pt}I})}
=\bigoplus_{i=1}^l \R \vep_i^{(\mathrm{I\hspace{-1.2pt}I})}.\]
Since 
\[\alpha_i^{(\mathrm{I\hspace{-1.2pt}I})}=\begin{cases}
\vep_i^{(\mathrm{I\hspace{-1.2pt}I})}-\vep_{i+1}^{(\mathrm{I\hspace{-1.2pt}I})} &
(1\leq i\leq l-1),\\
2\vep_l^{(\mathrm{I\hspace{-1.2pt}I})} & (i=l),
\end{cases}\]
we  have 
\begin{equation}
F_f^{(\mathrm{I\hspace{-1.2pt}I})}
=\bigoplus_{i=1}^l \R\alpha_i^{(\mathrm{I\hspace{-1.2pt}I})}
\quad\text{and}\quad
F=F_f^{(\mathrm{I\hspace{-1.2pt}I})}\bigoplus \R\delta.
\end{equation}

For $\sharp\in\{\mathrm{I},\mathrm{I\hspace{-1.2pt}I}\}$, 
let $\Lambda_0^{(\sharp)}$ be the element of $\fh^\ast$ which satisfies 
\begin{equation}\label{defn:Lambda0}
\widehat{I}\big((\alpha_i^{(\sharp)})^\vee,\Lambda_0^{(\sharp)}\big)
=\delta_{i,0} \quad (0\leq i\leq l),
\quad \widehat{I}(\Lambda_0^{(\sharp)},\Lambda_0^{(\sharp)})=0.
\end{equation}
These conditions determine $\Lambda_0^{(\sharp)}$ uniquely, and the explicit forms
of them are given by
\begin{equation}
\Lambda_0^{(\mathrm{\sharp})}=\begin{cases}
2\gamma & \text{if }\sharp=\mathrm{I},\\
\gamma+\dfrac12(\vep_1+\cdots+\vep_l)-\dfrac{l}{8}\delta & \text{if }
\sharp=\mathrm{I\hspace{-1.2pt}I}.
\end{cases}
\end{equation}
\begin{rem}
By definition, one has
\begin{equation}
\widehat{I}(\delta,\Lambda_0^{(\sharp)})=\begin{cases}
a_0^\vee\,(=2) & \text{it }\sharp=\mathrm{I},\\
a_l^\vee\,(=1) & \text{if }\sharp=\mathrm{I\hspace{-1.2pt}I}. 
\end{cases}
\end{equation}
\end{rem}
Hence, we have the following decompositions of $\fh^\ast$:
\begin{equation}\label{eq:decomposition-dual-Cartan}
\fh^\ast=F_{f,\C}^{(\sharp)}\bigoplus \big(\C\delta\bigoplus \C \Lambda_0^{(\sharp)}\big),
\end{equation}
where $F_{f,\C}^{(\sharp)}:=F_f^{(\sharp)}\otimes\C$. Set 
\begin{equation}\label{defn:c-d}
c=\kappa(\delta)\quad \text{and} \quad
d^{(\sharp)}
=\begin{cases}
a_0\kappa\big(\Lambda_0^{(\mathrm{I})}\big)=
\kappa\big(\Lambda_0^{(\mathrm{I})}\big) & \text{if }\sharp=\mathrm{I},\\[3pt]
a_l\kappa\big(\Lambda_0^{(\mathrm{I\hspace{-1.2pt}I})}\big)=
2\kappa\big(\Lambda_0^{(\mathrm{I\hspace{-1.2pt}I})}\big) & 
\text{if }\sharp=\mathrm{I\hspace{-1.2pt}I}.
\end{cases}
\end{equation}
Then, we have decompositions of $\mathfrak{h}$:
\begin{equation}\label{eq:decomposition-Cartan}
\mathfrak{h}=\Big(\bigoplus_{i=0}^l\C h_i^{(\sharp)}\Big)\bigoplus \C d^{(\sharp)}
=\Big(\bigoplus_{i=1}^l\C h_i^{(\sharp)}\Big)\bigoplus \C c\bigoplus \C d^{(\sharp)}.
\end{equation}

\subsubsection{}
For $\sharp\in\{\mathrm{I},\mathrm{I\hspace{-1.2pt}I}\}$, set 
\begin{equation}\label{defn:Delta_f}
\Delta_f^{(\sharp)}=\Delta^{re} \cap F_{f,\C}^{(\sharp)}. 
\end{equation}
In the following, we study the structure of $\Delta_f^{(\sharp)}$ in detail. 
\vskip 3mm
\noindent
\underline{\bf (i) Case of $\sharp=\mathrm{I}$.}
\vskip 1mm
By the explicit description \eqref{eq:real-roots} of $\Delta^{re}$, we have 
\[\Delta_f^{(\mathrm{I})}=(\Delta_f^{(\mathrm{I})})_s\amalg (\Delta_f^{(\mathrm{I})})_m 
\quad\text{where }
\begin{cases}
(\Delta_f^{(\mathrm{I})})_s:=\{\pm \vep_i\,|\,1\leq i\leq l\},\\[3pt]
(\Delta_f^{(\mathrm{I})})_m:=\{\pm \vep_i\pm\vep_j\,|\,1\leq i<j\leq l\}.
\end{cases}\]
This is the finite root system of type $B_l$ with a simple system 
$\Pi_f^{(\mathrm{I})}:=\{\alpha_1^{(\mathrm{I})},\ldots,\alpha_l^{(\mathrm{I})}\}$. 
\vskip 3mm
\noindent
\underline{\bf (ii) Case of $\sharp=\mathrm{I\hspace{-1.2pt}I}$.}
\vskip 1mm
By definition, the description \eqref{eq:real-roots} of $\Delta^{re}$ can be rewritten as
\begin{align*}
\Delta_s^{re}&=\big\{\pm \vep_i^{(\mathrm{I\hspace{-1.2pt}I})}
+(r+\tfrac12)\delta\,\big|\,1\leq i\leq l,\, r\in\Z\big\},\\
\Delta_m^{re}&=\big\{\pm \vep_i^{(\mathrm{I\hspace{-1.2pt}I})}\pm 
\vep_j^{(\mathrm{I\hspace{-1.2pt}I})}+r\delta\,\big|\,1\leq i<j\leq l,\, r\in\Z\big\},\\
\Delta_l^{re}&=\big\{\pm2 \vep_i^{(\mathrm{I\hspace{-1.2pt}I})}
+2r\delta\,\big|\,1\leq i\leq l,\, r\in\Z\big\}.
\end{align*}
Therefore, we have 
\[\Delta_f^{(\mathrm{I\hspace{-1.2pt}I})}
=(\Delta_f^{(\mathrm{I\hspace{-1.2pt}I})})_m\amalg 
(\Delta_f^{(\mathrm{I\hspace{-1.2pt}I})})_l 
\quad\text{where }
\begin{cases}
(\Delta_f^{(\mathrm{I\hspace{-1.2pt}I})})_m
:=\big\{\pm \vep_i^{(\mathrm{I\hspace{-1.2pt}I})}
\pm \vep_j^{(\mathrm{I\hspace{-1.2pt}I})}\,\big|\,1\leq i<j\leq l\big\},\\[3pt]
(\Delta_f^{(\mathrm{I\hspace{-1.2pt}I})})_l
:=\big\{\pm 2\vep_i^{(\mathrm{I\hspace{-1.2pt}I})} \,\big|\,1\leq i\leq l\big\}.
\end{cases}\]
This is the finite root system of type $C_l$ with a simple system 
$\Pi_f^{(\mathrm{I\hspace{-1.2pt}I})}:=\{\alpha_1^{(\mathrm{I\hspace{-1.2pt}I})},\ldots,
\alpha_l^{(\mathrm{I\hspace{-1.2pt}I})}\}$. \\

Let $\varpi_i^{(\sharp)}\in F_f^{(\sharp)} \,(1\leq i\leq l)$ be fundamental weights
of $\Delta_f^{(\sharp)}$ with respect to the simple system $\Pi_f^{(\sharp)}$.
They are characterized by 
$\widehat{I}\big(\varpi_i^{(\sharp)}, (\alpha_j^{(\sharp)})^\vee\big)=\delta_{i,j}$.
Define
\[\rho_f^{(\sharp)}=\sum_{i=1}^l \varpi_i^{(\sharp)}.\]
Hence, one has
\[\widehat{I}(\rho_f^{(\sharp)},(\alpha_j^{(\sharp)})^\vee)=1\quad 
\text{for every }1\leq j\leq l.\]
\begin{rem}
The explicit forms of them are given as follows{\rm :}
\[\varpi_i^{(\mathrm{I})}=
\begin{cases}
\sum_{j=1}^i\vep_j^{(\mathrm{I})} & (1\leq i\leq l-1),\\
\frac12\sum_{j=1}^l\vep_j^{(\mathrm{I})} & (i=l),
\end{cases}\quad
\varpi_i^{(\mathrm{I\hspace{-1.2pt}I})}
=\sum_{j=1}^i\vep_j^{(\mathrm{I\hspace{-1.2pt}I})}\ \ (1\leq i\leq l),\]  
\[\rho_f^{(\sharp)}=
\begin{cases}
\sum_{i=1}^l\tfrac{2(l-i)+1}{2}\vep_i^{(\mathrm{I})} & \text{if }\sharp=\mathrm{l},\\
\sum_{i=1}^l (l-i+1)\vep_i^{(\mathrm{I\hspace{-1.2pt}I})}
& \text{if }\sharp=\mathrm{I\hspace{-1.2pt}I}.
\end{cases}\]
\end{rem}
\subsection{The affine Weyl group $W$}
In this subsection, we study the structure of the affine Weyl group $W$. 
\subsubsection{}
Let $\mu\in F$ and a non-isotropic $\beta\in F$. Define linear automorphisms 
$t_\mu$ and $s_\beta$ of $\mathfrak{h}^\ast$ by
\[\begin{aligned}
t_\mu(u)&=u+\widehat{I}(u,\delta)\mu-\left\{\widehat{I}(u,\mu)
+\dfrac12 \widehat{I}(\mu,\mu)\widehat{I}(u,\delta)\right\}\delta,\\
s_\beta(u)&=u-\widehat{I}(u,\beta^\vee)\beta,
\end{aligned}\]
for $u\in \mathfrak{h}^\ast$, respectively. The following lemma is easily verified.
\begin{lemma}\label{lemma:translation}
{\rm (1)} Both $t_\mu$ and $s_\beta$ belong to 
$\mathrm{O}(\mathfrak{h}^\ast,\widehat{I})$.
\vskip 1mm
\noindent
{\rm (2)} For every $\mu_1,\mu_2\in F$, one has 
$t_{\mu_1}t_{\mu_2}=t_{\mu_1+\mu_2}$. 
\vskip 1mm
\noindent
{\rm (3)} For every non-isotropic $\beta\in F$, one has 
$s_{\delta-\beta}s_\beta=t_{\beta^\vee}$.
\vskip 1mm
\noindent
{\rm (4)} For $w\in W$ and $\mu\in F$, one has 
$w\circ t_\mu\circ w^{-1}=t_{w(\mu)}$. 
\end{lemma}

\subsubsection{}
For $\sharp\in\{\mathrm{I},\mathrm{I\hspace{-1.2pt}I}\}$, let us give a semidirect 
product description of $W$ attached to the choice of the special index.  
Recall the finite root subsystem $\Delta_f^{(\sharp)}$ of 
$\Delta^{re}$ introduced in \eqref{defn:Delta_f}, and let $W_f^{(\sharp)}$ be its 
Weyl group. Define a lattice $M^{(\sharp)}$ of rank $l$ as follows. 
\vskip 3mm
\noindent
(i) First, assume $\sharp=\mathrm{I}$. 
Let $\theta_s^{(\mathrm{I})}=\vep_1^{(\mathrm{I})}$ 
be the highest short root of $\Delta_f^{(\mathrm{I})}$. 
Since  $\alpha_0^{(\mathrm{I})}=\delta-2\theta_s^{(\mathrm{I})}$, we have
\[s_{\alpha_0^{(\mathrm{I})}}s_{\theta_s^{(\mathrm{I})}}=
s_{\delta-2\theta_s^{(\mathrm{I})}}s_{2\theta_s^{(\mathrm{I})}}=
t_{(2\theta_s^{(\mathrm{I})})^\vee}.\]
Therefore, the group $W$ is generated by 
$t_{(2\theta_s^{(\mathrm{I})})^\vee}$ and $W_f^{(\mathrm{I})}$. Set
\[M^{(\mathrm{I})}=\text{$\Z$-span of }W_f^{(\mathrm{I})}
\big((2\theta_s^{(\mathrm{I})})^\vee\big).\]
Since $(2\theta_s^{(\mathrm{I})})^\vee=\vep_1^{(\mathrm{I})}$ and 
$W_f^{(\mathrm{I})}$ is generated by $s_{\alpha_i^{(\mathrm{I})}}\, (1\leq i\leq l)$, 
we have
\[M^{(\mathrm{I})}=\bigoplus_{i=1}^l\Z \vep_i^{(\mathrm{I})}.\]
\vskip 3mm
\noindent
(ii) Second, assume $\sharp=\mathrm{I\hspace{-1.2pt}I}$. 
Let $\theta_l^{(\mathrm{I\hspace{-1.2pt}I})}=2\vep_1^{(\mathrm{I\hspace{-1.2pt}I})}$ 
be the highest (long) root of $\Delta_f^{(\mathrm{I\hspace{-1.2pt}I})}$. 
As $2\alpha_0^{(\mathrm{I\hspace{-1.2pt}I})}
=\delta-\theta_l^{(\mathrm{I\hspace{-1.2pt}I})}$,
we have the group $W$ is generated by 
$t_{(\theta_l^{(\mathrm{I\hspace{-1.2pt}I})})^\vee}$ and 
$W_f^{(\mathrm{I\hspace{-1.2pt}I})}$. Since 
$(\theta_l^{(\mathrm{I\hspace{-1.2pt}I})})^\vee=\vep_1^{(\mathrm{I\hspace{-1.2pt}I})}$ 
and $W_f^{(\mathrm{I\hspace{-1.2pt}I})}$ is generated by 
$s_{\alpha_i^{(\mathrm{I\hspace{-1.2pt}I})}}\, (1\leq i\leq l)$, we have 
\[M^{(\mathrm{I\hspace{-1.2pt}I})}
:=\text{$\Z$-span of }W_f^{(\mathrm{I\hspace{-1.2pt}I})}
\big((\theta_l^{(\mathrm{I\hspace{-1.2pt}I})})^\vee\big)
=\bigoplus_{i=1}^l\Z \vep_i^{(\mathrm{I\hspace{-1.2pt}I})}.\]

By Lemma \ref{lemma:translation}, we have the following lemma.
\begin{lemma}\label{lemma:aff-Weyl}
Set $t(M^{(\sharp)})=\{t_\mu\,|\,\mu\in M^{(\sharp)}\}$. This is an abelian
normal subgroup of $W$, and we have a semidirect
product description of $W${\rm :}
\[W=W_f^{(\sharp)}\ltimes t(M^{(\sharp)}).\] 
\end{lemma}
\subsection{The complex domain $\mathcal{Y}$ and its
coordinate systems}\label{sect:coordinate}
\subsubsection{}
Define a complex domain $\mathcal{Y}\subset \fh^\ast$ by  
\begin{equation}
\mathcal{Y}=\big\{v\in \fh^\ast\,\big|\, \text{Re}\,\widehat{I}(v,\delta)>0\big\}.
\end{equation}
Let $\H$ be the Poincer\'{e} upper half plane. 
For $\sharp=\{\mathrm{I},\mathrm{I\hspace{-1.2pt}I}\}$, define a map 
$\varphi^{(\sharp)}:\mathcal{Y}\longrightarrow Y:=\H\times \C^l\times \C$ as follows.
Let $\varphi_\delta,\ \varphi_i^{(\sharp)}\ (1\leq i\leq l),\ 
\varphi_{\Lambda_0}^{(\sharp)}$ be functions on $\mathcal{Y}$ 
defined by 
\[\varphi_\delta (v)=-\dfrac{\widehat{I}(v,\delta)}{2\pi\sqrt{-1}},\quad
\varphi_i^{(\sharp)}(v)=\dfrac{\widehat{I}(v,\vep_i^{(\sharp)})}{2\pi\sqrt{-1}},\quad
\varphi_{\Lambda_0}^{(\sharp)}(v)=\dfrac{\widehat{I}\Big(v,\tfrac{1}{(a_0^{(\sharp)})^\vee}
\Lambda_0^{(\sharp)}\Big)}{2\pi\sqrt{-1}},\]
for $v\in\mathcal{Y}$, where we set $(a_0^{(\sharp)})^\vee=a_0^\vee$ if $\sharp=\mathrm{I}$ and $(a_0^{(\sharp)})^\vee=a_l^\vee$ if 
$\sharp=\mathrm{I\hspace{-1.2pt}I}$.
Set
\[\varphi^{(\sharp)}:\mathcal{Y}\ni v \ \longmapsto \ 
\big(\varphi_\delta(v),\big(\varphi_1^{(\sharp)}(v),\ldots,\varphi_l^{(\sharp)}(v)\big),
\varphi_{\Lambda_0}^{(\sharp)}(v)\big)\in Y.\]
Since $\varphi^{(\sharp)}$ is an isomorphism of complex domains for each 
$\sharp=\{\mathrm{I},\mathrm{I\hspace{-1.2pt}I}\}$, 
we have the following commutative diagram:
\begin{center}
\begin{tikzpicture}[auto]
\node (a) at (1.5, 1.5) {$\mathcal{Y}$};  
\node (b) at (0,0) {$Y$};
\node (c) at (3, 0) {$Y$.};
\draw[->] (a) to node[swap] {$\scriptstyle \varphi^{(\mathrm{I})}$} (b);
\draw[->] (a) to node {$\scriptstyle \varphi^{(\mathrm{I\hspace{-1.2pt}I})}$} (c);
\draw[->] (0.25,0.1) to node 
{$\scriptstyle \varphi^{(\mathrm{I\hspace{-1.2pt}I})}\circ(\varphi^{(\mathrm{I})})^{-1}$} 
(2.75,0.1);
\draw[->] (2.75,-0.1) to node 
{$\scriptstyle \varphi^{(\mathrm{I})}\circ(\varphi^{(\mathrm{I\hspace{-1.2pt}I})})^{-1}$} 
(0.25,-0.1);
\draw (1.5,0.85) node {$\circlearrowright$};
\end{tikzpicture}
\end{center}

\begin{lemma}
Denote $(\tau,\underline{z},t)\in Y$ where $\tau\in\H, \underline{z}:=(z_1,\ldots,z_l)\in\C^l, t\in\C$. 
For every pair $(\sharp,\flat)$ such that 
$\sharp,\flat\in \{\mathrm{I},\mathrm{I\hspace{-1.2pt}I}\}$ and $\sharp\ne \flat$,
the explicit form of the map $\varphi^{(\sharp)}\circ (\varphi^{(\flat)})^{-1}:
Y\overset{\sim}{\longrightarrow}Y$ is given by
\begin{equation}\label{eq:transition}
(\tau,\underline{z},t)\longmapsto
\left(\tau,\left(-z_l-\dfrac12\tau,\ldots,-z_1-\dfrac12\tau\right),t+\dfrac{l}{8}\tau+
\dfrac12\sum_{i=1}^l z_i\right).
\end{equation}
\end{lemma}

\begin{proof}
First, we give the proof for the case that 
$\sharp=\mathrm{I}$ and $\flat=\mathrm{I\hspace{-1.2pt}I}$. 
For $(\tau,\underline{z},t)\in Y$, we have
\[(\varphi^{(\mathrm{I\hspace{-1.2pt}I})})^{-1}(\tau,\underline{z},t)=
2\pi\sqrt{-1}\left(-\tau \Lambda_0^{(\mathrm{I\hspace{-1.2pt}I})}
+\sum_{i=1}^l z_i \vep_i^{(\mathrm{I\hspace{-1.2pt}I})}+t \delta\right).\]
By definition, we have
{\small
\begin{align*}
&\text{\normalsize the right hand side}\\
&=2\pi\sqrt{-1}\left[-\tau
\left(\dfrac12\Lambda_0^{(\mathrm{I})}+\dfrac12\sum_{i=1}^l\vep_i^{(\mathrm{I})}
-\dfrac{l}{8}\delta
\right)
+\sum_{i=1}^l z_i \left(-\vep_{l+1-i}^{(\mathrm{I})}+\dfrac12\delta\right)+t\delta\right]\\
&=2\pi\sqrt{-1}\left[-\frac12 \tau \Lambda_0^{(\mathrm{I})}
+\sum_{i=1}^l\left(-z_{l+1-i}-\frac12\tau\right)\vep_i^{(\mathrm{I})}
+\left(t+\dfrac{l}{8}\tau+\frac12\sum_{i=1}^l z_i
\right)\delta\right].
\end{align*}}
Thus, we have the formula \eqref{eq:transition} in this case.

Second, by the above result, it follows immediately that
\[\big(\varphi^{(\mathrm{I})}\circ (\varphi^{(\mathrm{I\hspace{-1.2pt}I})})^{-1}\big)^2
=\mathrm{id}_Y.\]  
Therefore, we have 
\[\varphi^{(\mathrm{I\hspace{-1.2pt}I})}\circ (\varphi^{(\mathrm{I})})^{-1}
=\big(\varphi^{(\mathrm{I})}\circ (\varphi^{(\mathrm{I\hspace{-1.2pt}I})})^{-1}\big)^{-1}
=\varphi^{(\mathrm{I})}\circ (\varphi^{(\mathrm{I\hspace{-1.2pt}I})})^{-1}\]
as desired.
\end{proof}
\subsubsection{}
Introduce an isometry 
$\phi\in\mathrm{O}(\mathfrak{h}^\ast,\widehat{I})$ as follows. 
Define $w_0^{A_l}\in \mathrm{O}(\mathfrak{h}^\ast,\widehat{I})$ by
$w_0^{A_l}=s_{\alpha_1^{(\mathrm{I})}}
(s_{\alpha_2^{(\mathrm{I})}}s_{\alpha_1^{(\mathrm{I})}})\cdots
(s_{\alpha_{l-1}^{(\mathrm{I})}}\cdots s_{\alpha_1^{(\mathrm{I})}})$.
One has
\[w_0^{A_l}(\vep_i)=\vep_{l+1-i}\quad \text{for }1\leq i\leq l, \quad
w_0^{A_l}(\delta)=\delta\quad \text{and} \quad 
w_0^{A_l}(\Lambda_0^{(\mathrm{I})})=\Lambda_0^{(\mathrm{I})}\]
by direct computation. Introduce another isometry
$\zeta_{F_f}\in \mathrm{O}(\mathfrak{h}^\ast,\widehat{I})$ by
\[\zeta_{F_f}(\vep_i)=-\vep_i\quad \text{for }1\leq i\leq l, \quad
\zeta_{F_f}(\delta)=\delta\quad \text{and} \quad 
\zeta_{F_f}(\Lambda_0^{(\mathrm{I})})=\Lambda_0^{(\mathrm{I})}.\]
Set 
\begin{equation}
\phi= t_{(\vep_1+\cdots+\vep_l)/2}\circ w_0^A\circ \zeta_{F_f}\in
\mathrm{O}(\mathfrak{h}^\ast,\widehat{I}).
\end{equation}  
The next lemma follows form direct computation.  
\begin{lemma}
One has
\vskip 1mm
\noindent
{\rm (1)} $\phi(\vep_i^{(\mathrm{I})})=\vep_i^{(\mathrm{I\hspace{-1.2pt}I})}$ 
\ $(1\leq i\leq l)$,\ \ 
$\phi(\delta)=\delta$,\ \
$\phi(\Lambda_0^{(\mathrm{I})})=2\Lambda_0^{(\mathrm{I\hspace{-1.2pt}I})}$, 
\vskip 1mm
\noindent
{\rm (2)} $\phi$ is an involution.
\end{lemma}

Recall the decomposition \eqref{eq:decomposition-dual-Cartan}
of $\mathfrak{h}^*$, and let $\pi^{(\sharp)}:\mathfrak{h}^\ast \longrightarrow 
F_{f,\C}^{(\sharp)}\subset\mathfrak{h}^\ast$ be the canonical projection
for $\sharp\in\{\mathrm{I},\mathrm{I\hspace{-1.2pt}I}\}$.
The next lemma
is a direct consequence of the previous lemma.
\begin{lemma}
{\rm (1)} We have
\[\phi\circ \pi^{(\mathrm{I\hspace{-1.2pt}I})}= \pi^{(\mathrm{I})}\circ \phi,\quad 
\phi\circ  \pi^{(\mathrm{I})}=\pi^{(\mathrm{I\hspace{-1.2pt}I})}\circ \phi.\]
{\rm (2)} There is the following commutative diagram.
\begin{center}
\begin{tikzpicture}[auto]
\node (a) at (0,1.5) {$\mathcal{Y}$};  
\node (b) at (3.1,1.5) {$\mathcal{Y}$};
\node (c) at (1.5, 0) {$Y$};
\draw[->] (a) to node[swap] {$\scriptstyle \varphi^{(\mathrm{I})}$} (c);
\draw[->] (b) to node {$\scriptstyle \varphi^{(\mathrm{I\hspace{-1.2pt}I})}$} (c);
\draw[->] (0.25,1.6) to node {$\scriptstyle \phi$} (2.75,1.6);
\draw[->] (2.75,1.4) to node 
{$\scriptstyle \phi$} 
(0.25,1.4);
\draw (1.5,0.65) node {$\circlearrowright$};
\end{tikzpicture}
\end{center}

\end{lemma}
For later use, we set 
\begin{equation}\label{defn:pr-Y}
\mathrm{pr}^{(\sharp)}=\pi^{(\sharp)} \circ (\varphi^{(\sharp)})^{-1}.
\end{equation}
\subsection{Weight lattice}
For $1\leq j\leq l$ and $\sharp\in\{\mathrm{I},\mathrm{I\hspace{-1.2pt}I}\}$, define 
the $j$-th (affine) fundamental weight 
$\Lambda_j^{(\sharp)}$ (of type $(\sharp)$) by the condition that
\[\widehat{I}\big((\alpha_i^{(\sharp)})^\vee,\Lambda_j^{(\sharp)}\big)=\delta_{i,j}
\quad \text{for every }0\leq i\leq l.\] 
Note that this condition determines $\Lambda_j^{(\sharp)}$ modulo $\C\delta$. 

\begin{rem}\label{rem:fund-weights-1}
By direct computation, we have explicit forms of them as follows.
\[\Lambda_j^{(\mathrm{I})}\equiv 
\begin{cases}
\varpi_j^{(\mathrm{I})}+\Lambda_0^{(\mathrm{I})}\ \mathrm{mod}\, \C\delta & (1\leq j< l),\\
\varpi_l^{(\mathrm{I})}+\tfrac12 \Lambda_0^{(\mathrm{I})}\ \mathrm{mod}\, \C\delta & 
(j=l),
\end{cases}
\quad 
\Lambda_j^{(\mathrm{I\hspace{-1.2pt}I})}\equiv \Lambda_{l-j}^{(\mathrm{I})}\ 
\mathrm{mod}\, \C\delta.\]
\end{rem}
Define subsets $P$ and $P_+$ of $\mathfrak{h}^\ast$ by
\[P=\Big(\bigoplus_{j=0}^l \Z \Lambda_j^{(\mathrm{I})}\Big)+\C\delta\quad
\text{and}\quad
P_+=\Big(\sum_{j=0}^l \Z_{\geq 0} \Lambda_j^{(\mathrm{I})}\Big)+ \C\delta.\]
We call $P$ the weight lattice, and $P_+$ the set of dominant 
weights, respectively.

\begin{rem}
As we already mentioned in Remark \ref{rem:fund-weights-1}, we have
$\Lambda_j^{(\mathrm{I\hspace{-1.2pt}I})}\equiv \Lambda_{l-j}^{(\mathrm{i})}\ 
\mathrm{mod}\, \C\delta$. Therefore, we have
\[P=\Big(\bigoplus_{j=0}^l \Z \Lambda_j^{(\mathrm{I\hspace{-1.2pt}I})}\Big)+\C\delta\quad
\text{and}\quad
P_+=\Big(\sum_{j=0}^l \Z_{\geq 0} \Lambda_j^{(\mathrm{I\hspace{-1.2pt}I})}\Big)
+ \C\delta.\]
\end{rem}

For $\lambda\in P$, the number $k:=\widehat{I}(\delta,\lambda)$ is called
the \textbf{level} of $\lambda$. In the case of type $BC_l^{(2)}$, the levels
of fundamental weights are given in the following table.

 \begin{center}
    \begin{tabular}{|c||c|c|c|c|c|} \hline
  
  $j$ & $0$ & $1$ & $\cdots $ & $l-1$ & $l$ \\ \hline
   level of $\Lambda_j^{(\mathrm{I})}$ & $2$ & $2$ & $\cdots$ & $2$ & $1$  \\ \hline
   level of  $\Lambda_j^{(\mathrm{I\hspace{-1.2pt}I})}$ & 
   $1$ & $2$ & $\cdots$ & $2$ & $2$ \\ \hline  
    \end{tabular}
\end{center}
\vskip 3mm
\noindent
For $k\in\Z_{\geq 0}$, set 
\[P_k=\{\lambda\in P\,|\,\widehat{I}(\delta,\lambda)=k\}
\quad\text{and}\quad
P_{k,+}=P_k\cap P_+.
\]

For $\sharp\in \{\mathrm{I},\mathrm{I\hspace{-1.2pt}I}\}$, define 
$\rho^{(\sharp)}\in P_+$ by
\[\rho^{(\sharp)}=\sum_{j=0}^l\Lambda_j^{(\sharp)}.\]
Since fundamental weights can only be determined modulo $\C\delta$, 
$\rho^{(\sharp)}$ is determined modulo $\C \delta$ as well. Furthermore,  as  
$\Lambda_j^{(\mathrm{I\hspace{-1.2pt}I})}\equiv \Lambda_{l-j}^{(\mathrm{I})}\ 
\mathrm{mod}\, \C\delta$, we have
\[\rho^{(\mathrm{I\hspace{-1.2pt}I})}\equiv \rho^{(\mathrm{I})}\ 
\mathrm{mod}\, \C\delta.\]
That is, $\rho^{(\mathrm{I\hspace{-1.2pt}I})}$ coincides with 
$\rho^{(\mathrm{I})}$ as an affine weight. 

Let $\rho$ be any element of 
$\rho^{(\mathrm{I})}+\C \delta=\rho^{(\mathrm{II})}+\C \delta$. 
Note that, by definition, we have the level of $\rho$ is $2l+1$.

\section{Characters}\label{sect:characters}

\subsection{Algebra $\cE$, formal characters}
Recall the triangular decomposition \eqref{eq:triangular-decomp}
of $\fg$: $\fg=\fn_+ \oplus \fh \oplus \fn_-$.\\

Let $V$ be a $\fh$-diagonalizable module, i.e., 
$V=\bigoplus_{\lambda \in \fh^\ast} V_\lambda$ where 
$V_\lambda:=\{v \in V \, \vert \, 
h.v=\langle h, \lambda\rangle v \quad h \in \fh \, \}$. 
Set $\cP(V):=\{ \lambda \in \fh^\ast \vert V_\lambda\neq \{0\}\}$.
Let $\cO$ be the BGG category of $\fg$-modules, that is, it is the subcategory of 
$\fg$-modules whose objects are $\fh$-diagonalizable $\fg$-modules 
$V=\bigoplus_{\lambda \in \fh^\ast} V_\lambda$ satisfying
\begin{enumerate}
\item[(i)] $\dim V_\lambda< \infty$ for any $\lambda \in \cP(V)$, 
\item[(ii)] there exists $\lambda_1, \lambda_2, \cdots, \lambda_r \in \fh^\ast$ 
such that $\cP(V) \subset \bigcup_{i=1}^r (\lambda_i -\Z_{\geq 0}\Pi)$.
\end{enumerate}
A typical object of this category is a so-called highest weight module defined as 
follows. We say that a $\fg$-module $V$ is a highest weight module with highest weight 
$\Lambda \in \fh^\ast$ if
\begin{itemize}
\item[(i)] $\dim V_\Lambda=1$ 
\item[(ii)] $\fn_+.V_\Lambda=\{0\}$ and $V=U(\fg).V_\Lambda$.
\end{itemize}
In particular, the last condition implies that $V=U(\fn_-).V_\Lambda$ as a vector 
space, hence
\[ \cP(V) \subset 
\Lambda - \Z_{\geq 0}\Pi. \]
Such an example is given as follows. 
For $\Lambda \in \fh^\ast$, let $\C_\Lambda=\C v_\Lambda$ be the one 
dimensional module over $\fb_+:=\fh \oplus \fn_+$ defined by 
\[ h.v_\Lambda=\langle h, \Lambda \rangle v_\lambda \quad (h \in \fh), \qquad 
\fn_+.v_\Lambda=0. \]
The induced $\fg$-module  $M(\Lambda):=\Ind_{\fb_+}^{\fg}\C_\Lambda$ is called 
the Verma module with highest weight $\Lambda$. It can be shown that for any 
highest weight $\fg$-module $V$ with highest weight $\Lambda \in \fh^\ast$, 
there exists a surjective $\fg$-module map $M(\Lambda) \twoheadrightarrow V$. 
The smallest among such $V$ can be obtained by taking the quotient of 
$M(\Lambda)$ by its maximal proper submodule and the resulting $\fg$-module is 
the irreducible highest weight $\fg$-module with highest weight $\Lambda$, denoted by 
$L(\Lambda)$. 

Let $\cE$ be the formal linear combination of $e^{\lambda}$ 
($\lambda \in \fh^\ast$) with the next condition: 
$\sum_\lambda c_\lambda e^\lambda \in \cE$ $\Rightarrow$ 
$\exists \lambda_1, \lambda_2, \cdots, \lambda_r \in \fh^\ast$ such that
\[ \{ \lambda \vert c_\lambda \neq 0\} \subset 
\bigcup_{i=1}^r (\lambda_i - \Z_{\geq 0}\Pi). \]
We introduce the ring structure on $\cE$ by 
$e^{\lambda} \cdot e^{\mu}:=e^{\lambda+\mu}$. \\

The formal character of $V \in \cO$ is, by definition, the element $\ch V \in \cE$ 
defined by
\[ \ch V=\sum_{\lambda \in \cP(V)} (\dim V_\lambda)e^\lambda. \]
$\ch (\cdot )$ can be viewed as an additive function defined on $\cO$ with 
values in $\cE$. For example, The formal character of the Verma module 
$M(\Lambda)$ is given by
\[ \ch M(\Lambda)=e^{\Lambda} \prod_{\alpha \in \Delta_+}(1-e^{-\alpha})^{-
\mult(\alpha)}, \]
where $\Delta_+$ is the set of positive roots of $\fg$ and 
$\mult(\alpha)=\dim \fg_\alpha$ is the multiplicity of the root $\alpha$. 
Set $\vep(w)=(-1)^{l(w)}$ where $l(w)$ is the length of an element $w \in W$. 
For $\Lambda \in P_+$, the character of $L(\Lambda)$ is known 
as Weyl-Kac character formula and is given by
\begin{equation}\label{eqn:Weyl-Kac} \ch L(\Lambda)=\frac{\sum_{w \in W} \vep(w)e^{w(\Lambda+\rho)-\rho}}
{\prod_{\alpha \in \Delta_+}(1-e^{-\alpha})^{\mult(\alpha)}}, 
\end{equation}
where $\rho \in \fh^\ast$ is  a so-called Weyl vector, i.e., it satisfies $\langle h_i, 
\rho\rangle =1$ for any $0\leq i\leq l$. 
In particular, for $\Lambda=0$, as $L(0)=\C$ is the trivial representation, one 
obtains the so-called denominator identity:
\begin{equation}\label{denominator-id}
\sum_{w \in W} \vep(w) e^{w(\rho)}
=e^\rho \prod_{\alpha \in \Delta_+}(1-e^{-\alpha})^{\mult(\alpha)}. 
\end{equation}
This implies that the Weyl-Kac character formula can be rephrased as follows:
\begin{equation}\label{Weyl-Kac}
\ch L(\Lambda)=\frac{\sum_{w \in W} \vep(w)e^{w(\Lambda+\rho)}}
{\sum_{w \in W} \vep(w)e^{w(\rho)}}.
\end{equation}

\subsection{Nice functions on the Weyl group $W$ and twisted character}
Let $\Delta^{re}=\bigsqcup_{k\in S}O_k$ be the orbit space decomposition of
$\Delta^{re}$ with respect to the $W$-action with the set $S$ 
of parameters. 
Denote the index set of the simple system $\Pi$ by $I=\{0,1,\ldots,l\}$ and set
$I_k=\{i\in I\,|\,\alpha_i\in \Pi\cap O_k\}$. Then, we have a partition 
$I=\bigsqcup_{k\in S}I_k$ of the index set $I$ of $\Pi$.

For any map $\varphi: S \rightarrow \{ \pm 1\}$,  there exists a morphism of groups $\psi: W \twoheadrightarrow \{ \pm 1\}$ verifying 
$\psi(s_{\alpha_i})=\varphi(k)$ for any $i \in I_k$. 
This map is well-defined as  the Weyl group $W$ is a Coxeter group. Let $S_{odd}$ be the subset of $S$ consisting of those $k$ such that for any $i \in I_k$, $I(\alpha_i^\vee,\alpha_j) \in 2\Z$ for any $j \in I$. For a subset $S' \subset S_{odd}$, set $\tau=\bigcup_{k \in S'} I_k$.  Indeed, in our case, $\tau$ can be either $\emptyset$ or $\{l\}$.
The next two maps are typical examples of the above $\psi$:
\begin{enumerate}
\item $\vep(s_{\alpha_i})=-1$ \; $\forall\, i \in I$, 
\item $\vep'(s_{\alpha_i})=1$ for $i\in I \setminus \tau$  and $\vep'(s_{\alpha_i})=-1$ otherwise, i.e., for $i \in \tau$. 
\end{enumerate}
V. Kac introduced the next notion in pp. 100 of \cite{Kac1978}: a morphism of groups 
$\psi: W \rightarrow \{\pm 1\}$ is said to be \textbf{nice} if there exists a morphism of groups  $\bar{\psi}: \Z \Pi  \rightarrow \Z/2\Z$ satisfying 
$\psi(s_\alpha)=(-1)^{\bar{\psi}(\alpha)}$. A necessary and sufficient condition for $\psi$ to be nice is that $\psi(s_{\alpha_i})=-1$ for an 
$i \in I$ only if $i$-th line of the GCM $A$, i.e.,  $a_{i,j}$ is even for any $j \in I$. Hence, $\psi=1$ or $\psi=\vep'$ are examples of nice morphisms. 
\begin{rem} Regard the GCM $A$ as the GCM of the affine Lie superalgebra of type $B^{(1)}(0,l)$. Then, 
for the latter case, i.e., $\psi=\vep'$ with $\tau=\{l\}$, the map $\bar{\psi}$ coincides with the parity map $\vert \cdot \vert\,:\, Q(B^{(1)}(0,l)) \rightarrow \Z/2\Z$.
\end{rem}
In the rest of this article, we always consider the case when $A=(a_{i,j})$ is the GCM of type $BC_l^{(2)}$ and use the nice map $\psi$ 
defined as follows:
\begin{equation} \label{def_psi-map}
\psi: W \longrightarrow \{ \pm 1\}; \qquad s_{\alpha_i^{(\mathrm{I})}} \; \longmapsto \; \begin{cases} \;\; 1 \;\; & 0\leq i<l, \\ \; \; -1\; \; & i=l. \end{cases} 
\end{equation}
\begin{rem}\label{rem:psi-W} By definition, we have $\psi(s_{\alpha_i^{(\mathrm{I\hspace{-1.2pt}I})}})=(-1)^{\delta_{i,0}}$ which implies that $W_f^{(\mathrm{I\hspace{-1.2pt}I})} \subset \Ker \psi$ whereas $W_f^{(\mathrm{I})} \not\subset \Ker \psi$. Indeed, $W_f^{(\mathrm{I})} \cap \Ker \psi \subset W_f^{(\mathrm{I})}$ is a subgroup of index $2$, which is isomorphic to $W(D_l)$. We denote the subgroup $W_f^{(\mathrm{I})} \cap \Ker \psi$, which is generated by the reflections with respect to the middle length roots of $\Delta_f^{(\mathrm{I})}$, by $W_{f;m}^{(\mathrm{I})}$. 
\end{rem}

For a highest weight $\fg(A)$-module $V$ with highest weight $\Lambda$, we define its \textbf{twisted character} $\ch^{\psi}(V)$ as follows:
\[ \ch^{\psi}(V)=\sum_{\lambda \in \cP(V)} (-1)^{\bar{\psi}(\Lambda-\lambda)}(\dim V_\lambda)e^\lambda \in \cE, \]
where $V=\bigoplus_{\lambda \in \cP(V)} V_\lambda$ signifies its weight space decomposition. 
It has been essentially shown by V. G. Kac \cite{Kac1978} that, for an even level $\Lambda \in P_+$, one has
\begin{equation}\label{twisted-Weyl-Kac}
\ch^\psi L(\Lambda)=\frac{\sum_{w \in W}\vep(w)\psi(w)e^{w(\Lambda+\rho)}}{\sum_{w \in W}\vep(w)\psi(w)e^{w(\rho)}}.
\end{equation} 
\section{Theta functions}\label{sect:theta}
For $\lambda, \mu \in \fh^\ast$, set $e^{\lambda}(\mu)=e^{\widehat{I}(\lambda,\mu)}$. In this way, one can regard the formal (twisted-)characters as functions on a certain domain in $\fh^\ast$. Here, we introduce $\theta$-series and regard them as analytic functions. 
\subsection{Theta functions}
Let $k \in \Z_{> 0}$ be a positive integer. For $\lambda \in P_{k,+}$ and $\sharp \in \{\, \mathrm{I}, \mathrm{I\hspace{-1.2pt}I}\, \}$, we define the formal theta series ${}^{(\sharp)}\Theta_\lambda$ and its twisted form ${}^{(\sharp)}\Theta_\lambda^{\psi}$ as follows:
\begin{align*}
{}^{(\sharp)}\Theta_\lambda=
&e^{-\frac{\widehat{I}(\lambda,\lambda)}{2k}\delta}\sum_{\gamma \in M^{(\sharp)}} e^{t_\gamma(\lambda)} 
=
\sum_{\gamma \in M^{(\sharp)}} e^{\lambda+k\gamma-\frac{1}{2k}\widehat{I}(\lambda+k\gamma,\lambda+k\gamma)\delta}, \\
{}^{(\sharp)}\Theta_\lambda^{\psi}=
&e^{-\frac{\widehat{I}(\lambda,\lambda)}{2k}\delta}\sum_{\gamma \in M^{(\sharp)}} \psi(t_\gamma)e^{t_\gamma(\lambda)} 
=
\sum_{\gamma \in M^{(\sharp)}} \psi(t_\gamma)e^{\lambda+k\gamma-\frac{1}{2k}\widehat{I}(\lambda+k\gamma,\lambda+k\gamma)\delta},
\end{align*}
where $\psi:W \longrightarrow \{\pm 1\}$ is the nice function on $W$ defined in \eqref{def_psi-map}. They can be viewed as complex analytic functions on
the domain $Y$ defined in \S \ref{sect:coordinate}. 

By direct computations, one has, for $\sharp \in \{ \, \mathrm{I}, \mathrm{I\hspace{-1.2pt}I}\, \}$ and $\underline{y}=(\tau, \underline{z},t) \in Y$, 
\begin{align*}
&\big({}^{(\sharp)}\Theta_\lambda\circ (\varphi^{(\sharp)})^{-1}\big)(\tau,\underline{z},t)\\
=
&e^{2\widehat{I}(\lambda,\delta)\pi\sqrt{-1}t}\sum_{\gamma \in M^{(\mathrm{\sharp})}}
e^{\widehat{I}(\lambda,\delta)\pi\sqrt{-1}\tau \norm{\gamma+\frac{1}{k}\pi^{(\mathrm{\sharp})}(\lambda)}^2+2\widehat{I}(\lambda,\delta)\pi\sqrt{-1}\widehat{I}\big(\gamma+\frac{1}{k}\pi^{(\mathrm{\sharp})}(\lambda),\pr^{(\sharp)}(\underline{y})\big)}, \\
&\big({}^{(\sharp)}\Theta_\lambda^{\psi}\circ (\varphi^{(\sharp)})^{-1}\big)(\tau,\underline{z},t)\\
=
&e^{2\widehat{I}(\lambda,\delta)\pi\sqrt{-1}t}\sum_{\gamma \in M^{(\sharp)}}\psi(t_\gamma)
e^{\widehat{I}(\lambda,\delta)\pi\sqrt{-1}\tau \norm{\gamma+\frac{1}{k}\pi^{(\sharp)}(\lambda)}^2+2\widehat{I}(\lambda,\delta)\pi\sqrt{-1}\widehat{I}\big(\gamma+\frac{1}{k}\pi^{(\sharp)}(\lambda),\pr^{(\sharp)}(\underline{y})\big)},
\end{align*}
where $\pr^{(\sharp)}$ is defined in \eqref{defn:pr-Y}.
Fix $k \in \Z_{>0}$. 
Now, we introduce $W$-anti invariants, or equivalently, $\vep$-twisted $W$-invariants: for $\lambda \in P_{k,+}$, 
\[ A_{\lambda+\rho}:=e^{-\frac{\widehat{I}(\lambda+\rho,\lambda+\rho)}{2(k+2l+1)}\delta}\sum_{w \in W}\vep (w) e^{w(\lambda+\rho)}. \]
By Lemma \ref{lemma:aff-Weyl}  and the fact that $\vep(t_\gamma)=1$ for any $\gamma \in M^{(\sharp)}$, it follows that
\[ A_{\lambda+\rho}=\sum_{u \in W_f^{(\sharp)}} \vep(u) {}^{(\sharp)}\Theta_{u(\lambda+\rho)}, \]
for any $\sharp \in \{\, \mathrm{I}, \mathrm{I\hspace{-1.2pt}I}\, \}$. 
Similarly, we introduce $\vep\psi$-twisted $W$-invariants: for $\lambda \in P_{k,+} $, 
\[ A_{\lambda+\rho}^{\psi}:=e^{-\frac{\widehat{I}(\lambda+\rho,\lambda+\rho)}{2(k+2l+1)}\delta}\sum_{w \in W}\vep (w)\psi(w) e^{w(\lambda+\rho)}. \]
By Remark \ref{rem:psi-W}, we can rewrite it as follows:
\begin{align*}
 A_{\lambda+\rho}^{\psi}=
 &\sum_{u \in W_f^{(\sharp)}} \vep(u) \psi(u) {}^{(\sharp)}\Theta_{u(\lambda+\rho)}^{\psi} \\
  =
  &\begin{cases} \; \sum_{u \in W_{f;m}^{(\mathrm{I})}} \vep(u) \big({}^{(\mathrm{I})}\Theta_{u(\lambda+\rho)}^{\psi}+{}^{(\mathrm{I})}\Theta_{us_{\alpha_l^{(\mathrm{I})}}(\lambda+\rho)}^{\psi}\big)\; & \sharp=\mathrm{I}, \\
  \; \sum_{u \in W_f^{(\mathrm{I\hspace{-1.2pt}I})}} \vep(u) {}^{(\mathrm{I\hspace{-1.2pt}I})}\Theta_{u(\lambda+\rho)}^{\psi}\; & \sharp=\mathrm{I\hspace{-1.2pt}I}, 
  \end{cases}
\end{align*}
where the group $W_{f;m}^{(\mathrm{I})}$ is defined in Remark \ref{rem:psi-W}.
\begin{rem}\label{rem:twisted-denominator} For $\lambda=0$, one has the following \textbf{$($twisted-$)$denominator}$:$
\begin{align*}
A_\rho=&e^{\rho-\frac{\widehat{I}(\rho,\rho)}{2(2l+1)}\delta}\prod_{n=1}^\infty(1-e^{-n\delta})^l\prod_{\alpha \in (\Delta_f^{(\mathrm{I})})_s^+}(1-e^{-\alpha})\prod_{\alpha \in (\Delta_f^{(\mathrm{I})})_m^+}(1-e^{-\alpha}) \\
\times &\prod_{n=1}^\infty \Big(\prod_{\alpha \in (\Delta_f^{(\mathrm{I})})_s}(1-e^{-\alpha-n\delta})(1-e^{-2\alpha-(2n-1)\delta})\prod_{\alpha \in (\Delta_f^{(\mathrm{I})})_m}(1-e^{-\alpha-n\delta})\Big), \\
A_\rho^\psi=& e^{\rho-\frac{\widehat{I}(\rho,\rho)}{2(2l+1)}\delta}\prod_{n=1}^\infty(1-e^{-n\delta})^l\prod_{\alpha \in (\Delta_f^{(\mathrm{I})})_s^+}(1+e^{-\alpha})\prod_{\alpha \in (\Delta_f^{(\mathrm{I})})_m^+}(1-e^{-\alpha}) \\
\times &\prod_{n=1}^\infty \Big(\prod_{\alpha \in (\Delta_f^{(\mathrm{I})})_s}(1+e^{-\alpha-n\delta})(1-e^{-2\alpha-(2n-1)\delta})\prod_{\alpha \in (\Delta_f^{(\mathrm{I})})_m}(1-e^{-\alpha-n\delta})\Big).
\end{align*}
\end{rem}
When we regard the twisted invariants $A_{\lambda+\rho}$ and $A_{\lambda+\rho}^\psi$ as functions on $Y$, for $\underline{y}=(\tau,\underline{z},t) \in Y$, we have
\begin{align*}
\big(A_{\lambda+\rho}\circ (\varphi^{(\sharp)})^{-1})(\underline{y})=
&\sum_{u \in W_f^{(\sharp)}}\vep(u){}^{(\sharp)}\Theta_{u(\lambda+\rho)}(\underline{y}), \\
\big( A_{\lambda+\rho}^{\psi} \circ (\varphi^{(\sharp)})^{-1}\big)(\underline{y})=
 &\sum_{u \in W_f^{(\sharp)}}\vep(u)\psi(u){}^{(\sharp)}\Theta_{u(\lambda+\rho)}^{\psi}(\underline{y}).
\end{align*}
\begin{rem}\label{rem:Macdonald-id} An appropriate specialization of $A_\rho\circ (\varphi^{(\sharp)})^{-1}$ and $A_\rho^{\psi}\circ (\varphi^{(\sharp)})^{-1}$ $(\, \sharp \in \{\, \mathrm{I}, \mathrm{I\hspace{-1.2pt}I}\, \}\, )$ gives $\eta$-product identities appeared in the Appendix $1$, Type $BC_l$, $(6)$ in \cite{Macdonald1972}: \\
 \begin{center}
 \scalebox{0.9}{
    \begin{tabular}{|c||c|c|c|c|} \hline
  Twisted-invariant & $A_\rho\circ (\varphi^{(\mathrm{I})})^{-1}$ & $A_\rho\circ (\varphi^{(\mathrm{I\hspace{-1.2pt}I})})^{-1}$ & $A_\rho^{\psi}\circ (\varphi^{(\mathrm{I})})^{-1}$ &  $A_\rho^{\psi} \circ (\varphi^{(\mathrm{I\hspace{-1.2pt}I})})^{-1}$ \\ \hline
  Table in \cite{Macdonald1972} & $(a)$ & $(b)$ & $(c)$ & $(d)$  \\ \hline
    \end{tabular}}
\end{center}
\end{rem}
\subsection{Normalized (twisted-)characters} 
 For $\Lambda \in P_+$, the rational number
\[ c_\Lambda=\frac{\widehat{I}(\Lambda+\rho,\Lambda+\rho)}{2\widehat{I}(\Lambda+\rho,\delta)}-\frac{\widehat{I}(\rho,\rho)}{2\widehat{I}(\rho,\delta)}
\]
is called the \textbf{conformal anomaly}. We define the normalized (twisted-)character $\chi_\Lambda$ (resp. $\chi_\Lambda^\psi$) as follows:
\[ \chi_\Lambda:=e^{-c_{\Lambda}\delta}\ch L(\Lambda), \qquad \chi_\Lambda^\psi:=e^{-c_\Lambda \delta}\ch^\psi L(\Lambda). \]
It follows from the Weyl-Kac character formula \eqref{Weyl-Kac} and its twisted version \eqref{twisted-Weyl-Kac} that for an even level $\Lambda \in P_+$, one has
\begin{equation}\label{eqn:chi-A}
\chi_\Lambda=\frac{A_{\lambda+\rho}}{A_\rho}, \qquad \chi_\Lambda^\psi=\frac{A_{\Lambda+\rho}^\psi}{A_\rho^\psi}.
\end{equation} 

\section{Actions of $\mathrm{SL}_2(\Z)$}\label{sect:characters}
Recall the action of the group $\mathrm{SL}_2(\Z)$ on the complex domain $Y$: \; for $\begin{pmatrix} a & b \\ c & d \end{pmatrix} \in \mathrm{SL}_2(\Z)$ and $\underline{y}=(\tau,\underline{z},t) \in Y$, 
\[ \begin{pmatrix} a & b \\ c & d \end{pmatrix}.(\tau,\underline{z},t) =\Big( \frac{a\tau+b}{c\tau+d}, \frac{\underline{z}}{c\tau+d}, t-\frac{c\norm{\pr^{(\sharp)}(\underline{y})}^2}{2(c\tau+d)}\Big). \]
In particular, in this section, we study how the (twisted) $W$-invariant $\theta$-functions
behave under the action of $S:=\begin{pmatrix} 0 & -1 \\ 1 & 0 \end{pmatrix}$ and $T:=\begin{pmatrix} 1 & 1 \\ 0 & 1 \end{pmatrix}$. 
\subsection{Poisson's resummation formula}
Let $E=\R^l$ be an affine space and $I$ be a positive definite symmetric bilinear form on $E$. For a full sublattice $M$ of $E$, denote by $\mathrm{Vol}(M)$ the discriminant of $M$, i.e., the volume of $E/M$.  For a rapidly decreasing function $f: E \longrightarrow \C$ and $\underline{y} \in \R^l$, set
\[ \hat{f}(\underline{y}):=\int_{E} f(\underline{x})e^{-2\pi\sqrt{-1}I(\underline{x},\underline{y})}d\underline{x}. \]
The so-called Poisson's resummation formula is described as follows:
\begin{equation}\label{Poisson}
\sum_{ \underline{m} \in M} f(\underline{m})=\frac{1}{\mathrm{Vol}(M)^{\frac{1}{2}}}\sum_{\underline{m} \in M^\vee} \hat{f}(\underline{m}). 
\end{equation}
Here, 
\[ M^\vee:=\big\{ \, \underline{n} \in E\, \big\vert\, I(\underline{m}, \underline{n}) \in \Z, \; \forall\, \underline{m} \in M\, \big\} \]
is the dual lattice of $M$.  In particular, for a self-dual lattice $M$, $\underline{a} \in \C^l$, $\tau \in \H$ and $f(\underline{x}):=e^{\pi \sqrt{-1}
\big( -\frac{1}{\tau}\big)
\norm{\underline{x}+\underline{a}}^2}$, one obtains

\begin{equation}\label{Poisson_Heat-Kernel}
\sum_{ \underline{m} \in M} e^{\pi \sqrt{-1} 
\big( -\frac{1}{\tau}\big) 
\norm{\underline{m}+\underline{a}}^2}=
\Big( \frac{\tau}{\sqrt{-1}}\Big)^{\frac{1}{2}l}
\sum_{ \underline{m} \in M} e^{\pi \sqrt{-1} \tau \norm{\underline{m}}^2+2\pi \sqrt{-1}I(\underline{a},\underline{m})}. 
\end{equation}

See, e.g., \cite{Serre1970} for more information. 
\subsection{Modular transformations on twisted-invariants} \hskip 0.1in \\
As applications of \eqref{Poisson_Heat-Kernel}, we compute the modular transformation of the anti-invariants $A \circ (\varphi^{(\sharp)})^{-1}$ and $A^\psi \circ (\varphi^{(\sharp)})^{-1})$. Since the concrete computations are entirely similar to the cases studied by V. G. Kac and D. H. Peterson \cite{KacPeterson1984}, we omit the detail and write only the results. \\

For $\lambda, \mu \in P_{k,+}$, set
\begin{align*}
a^{(\mathrm{I})}(\lambda,\mu)=
&\sum_{u\in W_f^{(\mathrm{I})}}\vep(u)\psi(u)e^{-\frac{2\pi\sqrt{-1}}{k+2l+1}\cdot \widehat{I}\big(u\big(\pi^{(\mathrm{I})}(\lambda)+\rho_f^{(\mathrm{I})}\big),\pi^{(\mathrm{I})}(\mu)+\rho_f^{(\mathrm{I})}\big)}, \\
a^{(\mathrm{I}),(\mathrm{I\hspace{-1.2pt}I})}(\lambda,\mu)=
&\sum_{u \in W_f^{(\mathrm{I\hspace{-1.2pt}I})}}\vep(u)e^{-\frac{2\pi\sqrt{-1}}{k+2l+1}\cdot \widehat{I}\big(u\big(\pi^{(\mathrm{I\hspace{-1.2pt}I})}(\lambda)+\phi(\rho_f^{(\mathrm{I})})\big),\pi^{(\mathrm{I\hspace{-1.2pt}I})}(\mu)+\rho_f^{(\mathrm{I\hspace{-1.2pt}I})}\big)}, \\
a^{(\mathrm{I\hspace{-1.2pt}I}),(\mathrm{I})}(\lambda,\mu)=
&\sum_{u \in W_f^{(\mathrm{I})}}\vep(u)e^{-\frac{2\pi\sqrt{-1}}{k+2l+1}\cdot \widehat{I}\big(u\big(\pi^{(\mathrm{I})}(\lambda)+\phi(\rho_f^{(\mathrm{I\hspace{-1.2pt}I})})\big),\pi^{(\mathrm{I})}(\mu)+\rho_f^{(\mathrm{I})}\big)}, \\
a^{(\mathrm{I\hspace{-1.2pt}I})}(\lambda,\mu)=
&\sum_{u\in W_f^{(\mathrm{I\hspace{-1.2pt}I})}}\vep(u)e^{-\frac{2\pi\sqrt{-1}}{k+2l+1}\cdot \widehat{I}\big(u\big(\pi^{(\mathrm{I\hspace{-1.2pt}I})}(\lambda)+\rho_f^{(\mathrm{I\hspace{-1.2pt}I})}\big),\pi^{(\mathrm{I\hspace{-1.2pt}I})}(\mu)+\rho_f^{(\mathrm{I\hspace{-1.2pt}I})}\big)}. \\
\end{align*}

\begin{rem} Let $k \in 2\Z_{\geq 0}$. 
\begin{enumerate}
\item For any $\lambda, \mu \in P_{k,+}$, one has $a^{(\mathrm{I}),(\mathrm{I\hspace{-1.2pt}I})}(\phi(\lambda),\mu)=a^{(\mathrm{I\hspace{-1.2pt}I}),(\mathrm{I})}(\phi(\mu),\lambda)$. 
\item
The definition of $a^{(\mathrm{I})}(\lambda,\mu)$ can be rewritten as follows:
\begin{align*}
 a^{(\mathrm{I})}(\lambda,\mu)=
 &\sum_{u\in W_{f;m}^{(\mathrm{I})}}\vep(u)\Big(e^{-\frac{2\pi\sqrt{-1}}{k+2l+1}\cdot \widehat{I}\big(u\big(\pi^{(\mathrm{I})}(\lambda)+\rho_f^{(\mathrm{I})}\big),\pi^{(\mathrm{I})}(\mu)+\rho_f^{(\mathrm{I})}\big)} \\
&\phantom{\sum_{u\in W(D_l^{(\mathrm{I})}}\vep(u)}
+e^{-\frac{2\pi\sqrt{-1}}{k+2l+1}\cdot \widehat{I}\big(us_{\alpha_l^{(\mathrm{I})}}\big(\pi^{(\mathrm{I})}(\lambda)+\rho_f^{(\mathrm{I})}\big),\pi^{(\mathrm{I})}(\mu)+\rho_f^{(\mathrm{I})}\big)}\Big),
\end{align*}
where the group $W_{f;m}^{(\mathrm{I})}$ is defined in Remark \ref{rem:psi-W}.
\end{enumerate}
\end{rem}
The first lemma concerns the $W$-anti invariants and the coordinate system of type ($\mathrm{I}$) :
\begin{lemma}\label{lemma:Transf-A-(I)} Let $k \in 2\Z_{>0}$. For $\lambda \in P_{k,+}$ and $\underline{y}=(\tau,\underline{z},t) \in Y$, 
\begin{align*}
& \big(A_{\lambda+\rho} \circ (\varphi^{(\mathrm{I})})^{-1}\big)\Big( -\frac{1}{\tau}, \frac{\underline{z}}{\tau}, t-\frac{\norm{\pr^{(\mathrm{I})}(\underline{y})}^2}{2\tau}\Big) \\
=
&(k+2l+1)^{-\frac{1}{2}l}\Big( \frac{\tau}{\sqrt{-1}}\Big)^{\frac{1}{2}l} 
\sum_{\mu \in P_{k,+} \modo \C\delta}a^{(\mathrm{I}),(\mathrm{I\hspace{-1.2pt}I})}(\phi(\lambda),\mu)
\big(A_{\mu+\rho}^{\psi} \circ (\varphi^{(\mathrm{I\hspace{-1.2pt}I})})^{-1}\big)(\underline{y}), \\
&
\big(A_{\lambda+\rho}\circ (\varphi^{(\mathrm{I})})^{-1}\big)( \tau+1, \underline{z}, t) =
e^{\frac{\pi\sqrt{-1}}{k+2l+1}\norm{\pi^{(\mathrm{I})}(\lambda+\rho)}^2}
\big(A_{\lambda+\rho}\circ (\varphi^{(\mathrm{I})})^{-1}\big)(\underline{y}).
\end{align*}
In particular, 
\[
\big(A_{\rho}\circ (\varphi^{(\mathrm{I})})^{-1}\big)\Big( -\frac{1}{\tau}, \frac{\underline{z}}{\tau}, t-\frac{\norm{\pr^{(\mathrm{I})}(\underline{y})}^2}{2\tau}\Big)
=\Big( \frac{\tau}{\sqrt{-1}}\Big)^{\frac{1}{2}l} (\sqrt{-1})^{-l^2}\big(A_{\rho}^{\psi} \circ (\varphi^{(\mathrm{I\hspace{-1.2pt}I})})^{-1}\big)(\underline{y}).
\]
\end{lemma}
\vskip 0.2in
The second lemma concerns the $\vep\psi$-twisted $W$-invariants and the coordinate of type ($\mathrm{I}$) :
\begin{lemma}\label{lemma:Transf-A-psi-(I)} Let $k \in 2\Z_{>0}$. For $\lambda \in P_{k,+}$ and $\underline{y}=(\tau,\underline{z},t) \in Y$, 
\begin{align*}
& \big(A_{\lambda+\rho}^{\psi}\circ (\varphi^{(\mathrm{I})})^{-1}\big)\Big( -\frac{1}{\tau}, \frac{\underline{z}}{\tau}, t-\frac{\norm{\pr^{(\mathrm{I})}(\underline{y})}^2}{2\tau}\Big) \\
=
&(k+2l+1)^{-\frac{1}{2}l}\Big( \frac{\tau}{\sqrt{-1}}\Big)^{\frac{1}{2}l} \sum_{\mu \in P_{k,+} \modo \C\delta}a^{(\mathrm{I})}(\lambda,\mu)
\big(A_{\mu+\rho}^{\psi} \circ (\varphi^{(\mathrm{I})})^{-1}\big)(\underline{y}), \\
&
\big(A_{\lambda+\rho}^{\psi}\circ (\varphi^{(\mathrm{I})})^{-1}\big)( \tau+1, \underline{z}, t) =
e^{\frac{\pi\sqrt{-1}}{k+2l+1}\norm{\pi^{(\mathrm{I})}(\lambda+\rho)}^2}
\big(A_{\lambda+\rho}^{\psi}\circ (\varphi^{(\mathrm{I})})^{-1}\big)(\underline{y}).
\end{align*}
In particular, 
\[
\big(A_{\rho}^{\psi} \circ (\varphi^{(\mathrm{I})})^{-1}\big)\Big( -\frac{1}{\tau}, \frac{\underline{z}}{\tau}, t-\frac{\norm{\pr^{(\mathrm{I})}(\underline{y})}^2}{2\tau}\Big)
=\Big( \frac{\tau}{\sqrt{-1}}\Big)^{\frac{1}{2}l} (-1)^{\frac{1}{2}l(l-1)}\big(A_{\rho}^{\psi}\circ (\varphi^{(\mathrm{I})})^{-1}\big)(\underline{y}).
\]
\end{lemma}
\vskip 0.2in

The third lemma concerns the $W$-anti invariants and the coordinate of type ($\mathrm{I\hspace{-1.2pt}I}$) :
\begin{lemma}\label{lemma:Transf-A-(II)} Let $k \in 2\Z_{>0}$. For $\lambda \in P_{k,+}$ and $\underline{y}=(\tau,\underline{z},t) \in Y$, 
\begin{align*}
& \big(A_{\lambda+\rho}\circ (\varphi^{(\mathrm{I\hspace{-1.2pt}I})})^{-1}\big)\Big( -\frac{1}{\tau}, \frac{\underline{z}}{\tau}, t-\frac{\norm{\pr^{(\mathrm{I\hspace{-1.2pt}I})}(\underline{y})}^2}{2\tau}\Big) \\
=
&(k+2l+1)^{-\frac{1}{2}l}\Big( \frac{\tau}{\sqrt{-1}}\Big)^{\frac{1}{2}l} \sum_{\mu \in P_{k,+} \modo \C\delta}a^{(\mathrm{I\hspace{-1.2pt}I})}(\lambda,\mu)
\big(A_{\mu+\rho}\circ (\varphi^{(\mathrm{I\hspace{-1.2pt}I})})^{-1}\big)(\underline{y}), \\
&
\big(A_{\lambda+\rho}\circ (\varphi^{(\mathrm{I\hspace{-1.2pt}I})})^{-1}\big)( \tau+1, \underline{z}, t) =
e^{\frac{\pi\sqrt{-1}}{k+2l+1}\norm{\pi^{(\mathrm{I\hspace{-1.2pt}I})}(\lambda+\rho)}^2}
\big(A_{\lambda+\rho}^{\psi} \circ (\varphi^{(\mathrm{I\hspace{-1.2pt}I})})^{-1}\big)( \underline{y}).
\end{align*}
In particular, 
\[
\big(A_{\rho}\circ (\varphi^{(\mathrm{I\hspace{-1.2pt}I})})^{-1}\big)\Big( -\frac{1}{\tau}, \frac{\underline{z}}{\tau}, t-\frac{\norm{\pr^{\mathrm{I\hspace{-1.2pt}I})}(\underline{y})}^2}{2\tau}\Big)
=\Big( \frac{\tau}{\sqrt{-1}}\Big)^{\frac{1}{2}l} (\sqrt{-1})^{-l^2}\big(A_{\rho}\circ (\varphi^{(\mathrm{I\hspace{-1.2pt}I})})^{-1}\big)(\underline{y}).
\]
\end{lemma}
\vskip 0.2in
The fourth lemma concerns the $\vep\psi$-twisted $W$-invariants and the coordinate of type ($\mathrm{I\hspace{-1.2pt}I}$) :
\begin{lemma}\label{lemma:Transf-A-psi-(II)} Let $k \in 2\Z_{>0}$. For $\lambda \in P_{k,+}$ and $\underline{y}=(\tau,\underline{z},t) \in Y$, 
\begin{align*}
& \big(A_{\lambda+\rho}^{\psi}\circ (\varphi^{(\mathrm{I\hspace{-1.2pt}I})})^{-1}\big)\Big( -\frac{1}{\tau}, \frac{\underline{z}}{\tau}, t-\frac{\norm{\pr^{(\mathrm{I\hspace{-1.2pt}I})}(\underline{y})}^2}{2\tau}\Big) \\
=
&(k+2l+1)^{-\frac{1}{2}l}\Big( \frac{\tau}{\sqrt{-1}}\Big)^{\frac{1}{2}l} \sum_{\mu \in P_{k,+} \modo \C\delta}a^{(\mathrm{I\hspace{-1.2pt}I}),(\mathrm{I})}(\phi(\lambda),\mu)
\big( A_{\mu+\rho}\circ (\varphi^{(\mathrm{I})})^{-1}\big)(\underline{y}), \\
&\big( A_{\lambda+\rho}^{\psi}\circ (\varphi^{(\mathrm{I\hspace{-1.2pt}I})})^{-1}\big)( \tau+1, \underline{z}, t) =
e^{\frac{\pi\sqrt{-1}}{k+2l+1}\norm{\pi^{(\mathrm{I\hspace{-1.2pt}I})}(\lambda+\rho)}^2}
\big(A_{\lambda+\rho}\circ (\varphi^{(\mathrm{I\hspace{-1.2pt}I})})^{-1}\big)(\underline{y}).
\end{align*}
In particular, 
\[
\big(A_{\rho}^{\psi}\circ (\varphi^{(\mathrm{I\hspace{-1.2pt}I})})^{-1}\big)\Big( -\frac{1}{\tau}, \frac{\underline{z}}{\tau}, t-\frac{\norm{\pr^{(\mathrm{I\hspace{-1.2pt}I})}(\underline{y})}^2}{2\tau}\Big)
=\Big( \frac{\tau}{\sqrt{-1}}\Big)^{\frac{1}{2}l} (\sqrt{-1})^{-l^2}\big(A_{\rho}\circ (\varphi^{(\mathrm{I})})^{-1}\big)(\underline{y}).
\]
\end{lemma}

We remark the one of the key formulas in the proof of the above $4$ lemmas for $\lambda=0$ are the next well-known formula: for $N \in \Z_{>1}$, 
\[ \prod_{k=1}^{N-1} \sin \Big(\frac{k\pi}{N}\Big)=\frac{N}{2^{N-1}}. \]

\subsection{Modular transformations on (twisted-)characters} \hskip 0.1in \\
Let $k \in 2\Z_{\geq 0}$ and $\lambda \in P_{k,+}$.  Thanks to \eqref{eqn:chi-A}, the results stated in the previous section imply the following results: \\

By Lemma \ref{lemma:Transf-A-(I)}, we have
\begin{prop}\label{prop:Transf-chi-(I)} Let $k \in 2\Z_{>0}$. For $\lambda \in P_{k,+}$ and $\underline{y}=(\tau,\underline{z},t) \in Y$, 
\begin{align*}
&\big( \chi_\lambda \circ (\varphi^{(\mathrm{I})})^{-1}\big)\Big(-\frac{1}{\tau}, \frac{\underline{z}}{\tau}, t-\frac{\norm{\pr^{(\mathrm{I})}(\underline{y})}^2}{2\tau}\Big)\\
=
&(k+2l+1)^{-\frac{1}{2}l} (\sqrt{-1})^{l^2}\sum_{\mu \in P_{k,+} \modo \C \delta} a^{(\mathrm{I}),(\mathrm{I\hspace{-1.2pt}I})}(\phi(\lambda),\mu)\big(\chi_{\mu}^{\psi}\circ (\varphi^{(\mathrm{I\hspace{-1.2pt}I})})^{-1}\big)(\underline{y}), \\
&\big(\chi_\lambda\circ (\varphi^{(\mathrm{I})})^{-1}\big)(\tau+1, \underline{z}, t) \\
=
&e^{\pi\sqrt{-1}\Big(\frac{\norm{\pi^{(\mathrm{I})}(\lambda+\rho)}^2}{k+2l+1}-\frac{\norm{\pi^{(\mathrm{I})}(\rho)}^2}{2l+1}\Big)}
\big(\chi_\lambda\circ (\varphi^{(\mathrm{I})})^{-1}\big)(\underline{y}).
\end{align*}
\end{prop}
\vskip 0.2in
Lemma \ref{lemma:Transf-A-psi-(I)} implies
\begin{prop}\label{prop:Transf-chi-psi-(I)} Let $k \in 2\Z_{>0}$. For $\lambda \in P_{k,+}$ and $\underline{y}=(\tau,\underline{z},t) \in Y$, 
\begin{align*}
&\big(\chi_\lambda^{\psi} \circ (\varphi^{(\mathrm{I})})^{-1}\big)\Big(-\frac{1}{\tau}, \frac{\underline{z}}{\tau}, t-\frac{\norm{\pr^{(\mathrm{I})}(\underline{y})}^2}{2\tau}\Big)\\
=
&(k+2l+1)^{-\frac{1}{2}l} (-1)^{\frac{1}{2}l(l-1)}\sum_{\mu \in P_{k,+} \modo \C \delta} a^{(\mathrm{I})}(\lambda,\mu)\big(\chi_{\mu}^{\psi}\circ (\varphi^{(\mathrm{I})})^{-1}\big) (\underline{y}), \\
&\big(\chi_\lambda^{\psi}\circ (\varphi^{(\mathrm{I})})^{-1}\big)(\tau+1, \underline{z}, t) \\
=
&e^{\pi\sqrt{-1}\Big(\frac{\norm{\pi^{(\mathrm{I})}(\lambda+\rho)}^2}{k+2l+1}-\frac{\norm{\pi^{(\mathrm{I})}(\rho)}^2}{2l+1}\Big)}
\big(\chi_\lambda^{\psi}\circ (\varphi^{(\mathrm{I})})^{-1}\big)(\underline{y}).
\end{align*}
\end{prop}
\vskip 0.2In
By Lemma \ref{lemma:Transf-A-(II)}, we have (cf. the first identity had been obtained in \cite{KacPeterson1984})
\begin{prop}\label{prop:Transf-chi-(II)} Let $k \in 2\Z_{>0}$. For $\lambda \in P_{k,+}$ and $\underline{y}=(\tau,\underline{z},t) \in Y$, 
\begin{align*}
&\big(\chi_\lambda \circ (\varphi^{(\mathrm{I\hspace{-1.2pt}I})})^{-1}\big)\Big(-\frac{1}{\tau}, \frac{\underline{z}}{\tau}, t-\frac{\norm{\pr^{(\mathrm{I\hspace{-1.2pt}I})}(\underline{y})}^2}{2\tau}\Big)\\
=
&(k+2l+1)^{-\frac{1}{2}l} (\sqrt{-1})^{l^2}\sum_{\mu \in P_{k,+} \modo \C \delta} a^{(\mathrm{I\hspace{-1.2pt}I})}(\lambda,\mu)\big( \chi_{\mu} \circ (\varphi^{(\mathrm{I\hspace{-1.2pt}I})})^{-1}\big)(\underline{y}), \\
&\big(\chi_\lambda \circ (\varphi^{(\mathrm{I\hspace{-1.2pt}I})})^{-1}\big)(\tau+1, \underline{z}, t) \\
=
&e^{\pi\sqrt{-1}\Big(\frac{\norm{\pi^{(\mathrm{I\hspace{-1.2pt}I})}(\lambda+\rho^{(\mathrm{I\hspace{-1.2pt}I})})}^2}{k+2l+1}-\frac{\norm{\pi^{(\mathrm{I\hspace{-1.2pt}I})}(\rho^{(\mathrm{I\hspace{-1.2pt}I})})}^2}{2l+1}\Big)}
\big(\chi_\lambda^{\psi}\circ (\varphi^{(\mathrm{I\hspace{-1.2pt}I})})^{-1}\big)(\underline{y}).
\end{align*}
\end{prop}
Lemma \ref{lemma:Transf-A-psi-(II)} implies
\begin{prop}\label{prop:Transf-chi-psi-(II)} Let $k \in 2\Z_{>0}$. For $\lambda \in P_{k,+}$ and $\underline{y}=(\tau,\underline{z},t) \in Y$, 
\begin{align*}
&\big(\chi_\lambda^{\psi}\circ (\varphi^{(\mathrm{I\hspace{-1.2pt}I})})^{-1}\big)\Big(-\frac{1}{\tau}, \frac{\underline{z}}{\tau}, t-\frac{\norm{\pr^{(\mathrm{I\hspace{-1.2pt}I})}(\underline{y})}^2}{2\tau}\Big)\\
=
&(k+2l+1)^{-\frac{1}{2}l} (\sqrt{-1})^{l^2}\sum_{\mu \in P_{k,+} \modo \C \delta} a^{(\mathrm{I\hspace{-1.2pt}I}), (\mathrm{I})}(\phi(\lambda),\mu)\big(\chi_{\mu}\circ( \varphi^{(\mathrm{I})})^{-1}\big)(\underline{y}), \\
&\big(\chi_\lambda^{\psi}\circ (\varphi^{(\mathrm{I\hspace{-1.2pt}I})})^{-1}\big)(\tau+1, \underline{z}, t)\\
=
&e^{\pi\sqrt{-1}\Big(\frac{\norm{\pi^{(\mathrm{I\hspace{-1.2pt}I})}(\lambda+\rho^{(\mathrm{I\hspace{-1.2pt}I})})}^2}{k+2l+1}-\frac{\norm{\pi^{(\mathrm{I\hspace{-1.2pt}I})}(\rho^{(\mathrm{I\hspace{-1.2pt}I})})}^2}{2l+1}\Big)}
\big(\chi_\lambda \circ (\varphi^{(\mathrm{I\hspace{-1.2pt}I})})^{-1}\big)(\underline{y}).
\end{align*}
\end{prop}

\section{Conclusion and Observation}\label{sect:observations}
In 1984, V. Kac and D. Peterson \cite{KacPeterson1984} showed that there exists an action of the subgroup $\Gamma_\theta$ of $\mathrm{SL}_2(\Z)$, generated by $S$ and $T^2$, on the vector space
\[ V_k^{(\mathrm{I\hspace{-1.2pt}I})}:=\bigoplus_{\lambda \in P_{k ,+} \modo \C \delta} \C \big(\chi_{\lambda}\circ (\varphi^{(\mathrm{I\hspace{-1.2pt}I})})^{-1}\big).  \]
As $\Gamma_\theta$ is a subgroup of  $\mathrm{SL}_2(\Z)$ of index $3$, the dimension of the $\mathrm{SL}_2(\Z)$-module 
\[ V_k:=\Ind_{\Gamma_\theta}^{\mathrm{SL}_2(\Z)} V_k^{(\mathrm{I\hspace{-1.2pt}I})} \]
is $3 \dim V_k^{(\mathrm{I\hspace{-1.2pt}I})}$. Let $V_k^{(\mathrm{I})}$ and $V_k^{\psi,(\mathrm{I\hspace{-1.2pt}I})}$ be the vector spaces generated by 
$\chi_\lambda\circ (\varphi^{(\mathrm{I})})^{-1}$ ($\lambda \in P_{k,+} \mod \C \delta$) and $\chi_{\lambda}^{\psi} \circ (\varphi^{(\mathrm{I\hspace{-1.2pt}I})})^{-1}$ ($\lambda \in P_{k,+} \mod \C \delta$), respectively. 
From the previous section, we see that two vector spaces $V_k^{(\mathrm{I})}$ and $V_k^{\psi,(\mathrm{I\hspace{-1.2pt}I})}$ are subspaces of $V_k$ and that 
\[ V_k=V_k^{(\mathrm{I})} \oplus V_k^{(\mathrm{I\hspace{-1.2pt}I})} \oplus V_k^{\psi,(\mathrm{I\hspace{-1.2pt}I})}. \]
Indeed, by Propositions \ref{prop:Transf-chi-(I)}, \ref{prop:Transf-chi-(II)} and \ref{prop:Transf-chi-psi-(II)}, the $\mathrm{SL}_2(\Z)$-action on $V_k$, more precisely, the elements $S$ 
and $T$ map 
\begin{align*}
S:&  V_k^{(\mathrm{I\hspace{-1.2pt}I})} \righttoleftarrow 
V_k^{(\mathrm{I\hspace{-1.2pt}I})}, \qquad V_k^{(\mathrm{I})} \longrightarrow 
V_k^{\psi,(\mathrm{I\hspace{-1.2pt}I})}, \quad V_k^{\psi,(\mathrm{I\hspace{-1.2pt}I})} 
\longrightarrow V_k^{(\mathrm{I})} \\
T:&  V_k^{(\mathrm{I})} \righttoleftarrow V_k^{(\mathrm{I})}, \qquad 
V_k^{(\mathrm{I\hspace{-1.2pt}I})} \longrightarrow 
V_k^{\psi,(\mathrm{I\hspace{-1.2pt}I})}, \quad 
V_k^{\psi,(\mathrm{I\hspace{-1.2pt}I})} \longrightarrow V_k^{(\mathrm{I\hspace{-1.2pt}I})}.
\end{align*} 
Therefore, we have
\begin{thm} Fix $k \in 2\Z_{\geq 0}$. Then, $\mathrm{SL}_2(\Z)$ acts on the $\C$-vector space spanned by 
$ \big\{\; \chi_\lambda\circ (\varphi^{(\mathrm{I})})^{-1}, \; \chi_\lambda\circ (\varphi^{(\mathrm{I\hspace{-1.2pt}I})})^{-1}, \; 
\chi_\lambda^{\psi} \circ (\varphi^{(\mathrm{I\hspace{-1.2pt}I})})^{-1}\; \big\}_{\lambda \in P_{k,+} \modo \C \delta\;}.
$
\end{thm}
Set
\[ V_k^{\psi,(\mathrm{I})}=\bigoplus_{\lambda \in P_{k ,+} \modo \C \delta} \C \chi_{\lambda}^{\psi} \circ (\varphi^{(\mathrm{I})})^{-1}.  \]
Surprisingly, by Proposition \ref{prop:Transf-chi-psi-(I)}, we have the following theorem:
\begin{thm} Fix $k \in 2\Z_{\geq 0}$. Then, $\mathrm{SL}_2(\Z)$ acts on the $\C$-vector space spanned by $\big\{\; \chi_{\lambda}^{\psi} \circ (\varphi^{(\mathrm{I})})^{-1}\; \big\vert\; \lambda \in P_{k,+} \mod \C \delta\; \big\}$.  
\end{thm}
For an \textit{odd} $k$, the authors are not aware of any such interpretation. \\

There is an interesting connection between this theorem and the affine Lie superalgebra 
$\fg$ of type $B^{(1)}(0,l)$. This is a Lie superalgebra whose set of real roots, forgetting 
the $\Z/2\Z$-structure, is the non-reduced affine root system of type $BCC_l$. 
By V. Kac \cite{Kac1978}, it appears that both $\chi_\lambda\circ (\varphi^{(\mathrm{I})})^{-1}$ 
($\lambda \in P_{k,+}$) and $\chi_\lambda\circ (\varphi^{(\mathrm{I\hspace{-1.2pt}I})})^{-1}$ 
($\lambda \in P_{k,+}$) can be also viewed as the character of integrable highest 
weight modules of level $k$ over $\fg$, since the non-negative integer $k$ is supposed 
to be even. With this identification, the twisted characters 
$\chi_\lambda^{\psi} \circ (\varphi^{(\mathrm{I})})^{-1}$ and $\chi_\lambda^{\psi}\circ (\varphi^{(\mathrm{I\hspace{-1.2pt}I})})^{-1}$ 
can be viewed as the \textbf{super-character} of integrable highest weight modules of level $k$ over $\fg$ (cf. \S \ref{sect:IFDR}). There is $4$ type of affine Lie superalgebras whose universal enveloping algebras are integral domain. In addition, the set of real roots of such affine Lie superalgebras, forgetting their $\Z/2\Z$-graded structure, are non-reduced affine root systems: \\

\begin{center}
\begin{tabular}{|c||c|c|c|c|} \hline 
Non-reduced type & $BCC_l$ & $C^\vee BC_l$ & $BB_l^\vee$ & $C^\vee C_l$  \\ \hline 
Affine super          & $B^{(1)}(0,l)$ & $A^{(4)}(0,2l)$ & $A^{(2)}(0,2l-1)$ & $C^{(2)}(l+1)$ \\ \hline
\end{tabular}
\end{center}
\vskip 0.1in
This article deals with $BC_l^{(2)}$ which is connected with $B^{(1)}(0,l)$. For $3$ other cases, there might be something similar, as can be seen from the appendix of \cite{Macdonald1972}. The $\mathrm{SL}_2(\Z)$-symmetry, we have shown in this article, is quite mysterious and its $\ll$ raison d'être $\gg$ are to be revealed.

\renewcommand{\thesection}{\Alph{section}}
\appendix


\section{Lie superalgebra }\label{sect:IFDR}
Here, we recall some relevant basic facts about Lie superalgebras.  The basic reference here is \cite{Kac1978}. 

\subsection{Lie superalgebra $\mathfrak{osp}(1\vert2)$ and its irreducible representations}
The Lie superalgebra $\mathfrak{osp}(1\vert2)$ is the superspace
\[ \fg=\fg_{\bar{0}} \oplus \fg_{\bar{1}}, \qquad \fg_{\bar{0}}:=\C E \oplus \C H \oplus \C F, \qquad \fg_{\bar{1}}:=\C e \oplus \C f \]
equipped with the Lie super-bracket $[\, \cdot, \cdot\, ]$ satisfying
\begin{align*}
& [H, E]=4E, \qquad [H,F]=-4F, \qquad [E,F]=2H, \\
& [H,e]=2e, \qquad [H,f]=-2f, \\
& [e,f]=H, \qquad [e,e]=2E, \qquad [f,f]=-2F.
\end{align*}
Let $\alpha \in (\C H)^\ast$ such that $\alpha(H)=2$. Then, it is clear that $\{ \pm \alpha \} \cup \{ \pm 2\alpha \}$ is the set of roots of $\fg$ with respect to $\fh=\C H$. 
Set $\fb=\C H \oplus \C e \oplus \C E$. For $\lambda \in \fh^\ast$, let $\C v_{\lambda} $ be the even $\fb$-module defined by
\[ H.v_\lambda=\lambda (H)v_\lambda, \qquad e.v_\lambda=0. \]
It is clear that this implies $E.v_\lambda=0$. The $\fg$-module $M(\lambda):=\mathrm{Ind}_{\fb}^{\fg} \C v_\lambda$ is the Verma module with highest weight $\lambda$. 
For $i \in \Z_{\geq 0}$, set $w_i=\dfrac{1}{i!}f^{i}\otimes v_\lambda$. By direct computation, one sees 
\begin{align*}
& H.w_i=(\lambda(H)-2i)w_i, \qquad e.w_i=\begin{cases}\; -w_{i-1} \;&  i \equiv 0 \mod 2, \\ \; i(\lambda(H)-i+1)w_{i-1} \; & i\equiv 1 \mod 2, \end{cases} \\
&  f.w_i=(i+1)w_{i+1}. 
\end{align*}
Hence, the $\fg$-module $M(\lambda)$ is reducible if and only if $\lambda(H) \in 2\Z_{\geq 0}$, in which case the subspace $N(\lambda):=\bigoplus_{i>\lambda(H)} \C w_i$ gives the maximal proper $\fg$-submodule. Denote by $L(\lambda)$ its irreducible quotient. 
The next lemma is important:
\begin{lemma}\label{lemma:IFDR-osp}
 Let $V$ be an irreducible finite dimensional $\fg$-module. Then, there exists $N \in \Z_{\geq 0}$ such that  $V \cong L(N\alpha)$. In particular, $V$ is of odd dimension. 
 \end{lemma}
\subsection{Lie superalgebra of type $B^{(1)}(0,l)$}
Fix a finite set $I$. Recall that a Kac Moody Lie algebra is a Lie algebra associated with a GCM $A=(a_{i,j})_{i,j \in I}$, i.e., a square matrix with $a_{i,j} \in \Z$ such that i) $a_{i,i}=2$ for any $i \in I$, ii) $a_{i,j}\leq 0$ for any $i \neq j$ and iii) $a_{i,j}=0$ implies $a_{j,i}=0$. 
Its building blocks are isomorphic to $\fsl(2)$. Similarly,  a class of Lie superalgebras had been considered in the article  \cite{Kac1978} associated with a pair $(A, \tau)$ of a GCM $A$ and the subset $\tau \subset I$ satisfying the additional condition iv) $a_{i,j} \in 2\Z$ for any $i \in \tau$ and $j \in I$. Its building blocks are isomorphic either to $\fsl(2)$ or to $\mathfrak{osp}(1\vert2)$. 
 
Recall the definition of the Lie superalgebra $\fg(A,\tau)$. Let $(\fh, \Pi, \Pi^\vee)$ be a realization of the GCM $A$ so that $\Pi=\{ \alpha_i \}_{i \in I} \subset \fh^\ast$ and $\Pi^\vee=\{ h_i\}_{i \in I}$ are sets of linearly independent vectors satisfying $\langle \alpha_i, h_j\rangle=a_{j,i}$ for any $i,j \in I$. 
$\fg(A,\tau)$ is the Lie superalgebra generated by $e_i, f_i$ ($i \in I$) and $\fh$ with the parity
 \[ \vert e_i \vert=\vert f_i\vert := \begin{cases} \; \bar{0}\; & i \in I \setminus \tau, \\ 
                                                                        \; \bar{1}\; & i \in \tau, \end{cases} \qquad \vert h\vert:=\bar{0} \quad h \in \fh, \]
equipped with the Lie superbracket $[\cdot, \cdot]$ satisfying
\begin{itemize}
\item[(S1)] $[h,h']=0$ for $h,h'\in \mathfrak{h}$,
\vskip 1mm
\item[(S2)] $\left\{\begin{array}{l}
[h,e_i]=\langle h,\alpha_i\rangle \, e_i,\\[3pt] 
\lbrack h,f_i\rbrack=-\langle h,\alpha_i\rangle \, f_i,
\end{array}\right.$ for
$h\in\mathfrak{h}$ and $0\leq i\leq l$,
\vskip 1mm
\item[(S3)] $[e_i,f_j]=\delta_{i,j}h_i$ for $0\leq i,j\leq l$,
\vskip 1mm
\item[(S4)] $\left\{\begin{array}{l}
\mathrm{ad}\big(e_i\big)^{1- a_{i,j}}\big(e_j\big)=0,\\[3pt] 
\mathrm{ad}\big(f_i\big)^{1- a_{i,j}}\big(f_j\big)=0,
\end{array}\right.$ for $0\leq i\ne j\leq l$. 
\end{itemize} 
A non-zero element $\alpha \in \fh^\ast$ is called a root if 
\[ \fg_\alpha:=\big\{\, x\in \fg(A,\tau)\, \big\vert\, [h,x]=\alpha(h)x \; \forall\, h \in \fh\, \big\} \neq \{0\}. \]
We denote the set of all roots by $\Delta$. We have the so-called root space decomposition with respect to $\fh$: 
$\fg(A,\tau)=\fh \oplus \bigoplus_{\alpha \in \Delta} \fg_\alpha.$
A root $\alpha \in \Delta$ is called \textit{even} (resp. \textit{odd}) if any non-zero element of $\fg_\alpha$ is even (resp. \textit{odd}). 
It is clear that a root $\alpha=\sum_{i \in I}m_i\alpha_i \in \Delta$ is even iff so is $\sum_{i \in \tau}m_i$.  Denote by 
$\Delta_{\bar{0}}$ and $\Delta_{\bar{1}}$ the set of even and odd roots, respectively. By definition, one has the following parity decomposition:
\[
\fg(A,\tau)=\fg(A,\tau)_{\bar{0}} \oplus \fg(A,\tau)_{\bar{1}}, 
\]
where $\fg(A,\tau)_{\bar{0}}=\fh \oplus \bigoplus_{\alpha \in \Delta_{\bar{0}}}\fg_\alpha$ and $\fg(A,\tau)=\bigoplus_{\alpha \in \Delta_{\bar{1}}} \fg_\alpha$. For $i \in I$, Let $s_{\alpha_i}$ be the reflection in $\fh^\ast$ defined by $s_{\alpha_i}(x)=x-x(h_i)\alpha_i$ and $W(A)$ be the subgroup of $GL(\fh^\ast)$ generated by $s_{\alpha_i}$ ($i \in I$). An element of $\Delta^{re}:=W.\Pi$ (resp. $\Delta^{im}:=\Delta \setminus \Delta^{re}$) is called a \textit{real} (resp. an \textit{imaginary}) root. \\

The finite Lie superalgebra $\fg_f$ of type $B(0,l)$ is, by definition, the Lie superalgebra $\fg(A_f,I_{f,odd})$ associated with the Cartan matrix $A_f=(a_{i,j})_{1\leq i,j\leq l}$ of type $B_l$, i.e., 
\[ A_f={\small \begin{pmatrix} 2 & -1 & 0 & \cdots & 0  \\ 
                                   -1 & 2 & -1 &\ddots & \vdots  \\
                                   0  & \ddots & \ddots & \ddots & 0 \\
                                   \vdots & \ddots & -1 & 2 & -1 \\
                                   0 & \cdots & 0 & -2 & 2 \end{pmatrix}}, \qquad I_{f,odd}=\{l\} \subset I_f=\{1,2,\cdots, l\}.  \]
This Lie superalgebra is also called $\mathfrak{osp}(1\vert 2l)$. 
\begin{rem} The root system $\Delta_f$ of the Lie superalgebra $\fg_f$ is of type $BC_l$, i.e., $\Delta_f=(\Delta_f)_s \cup (\Delta_f)_m \cup (\Delta_f)_l$ where
\begin{align*}
&(\Delta_f)_s:=\{ \, \pm \vep_i\, \vert\, 1\leq i\leq l\, \}, \qquad (\Delta_f)_l:=\{ \, \pm 2\vep_i\, \vert\, 1\leq i\leq l\, \}, \\
&(\Delta_f)_m:=\{\, \pm \vep_i \pm \vep_j\, \vert\, 1\leq i<j\leq l\,\}. 
\end{align*}
Its parity decomposition is given as follows: $\Delta_f=(\Delta_f)_{\bar{0}} \amalg (\Delta_f)_{\bar{1}}$ where
\[ (\Delta_f)_{\bar{0}}=(\Delta_f)_m \amalg (\Delta_f)_l, \qquad (\Delta_f)_{\bar{1}}=(\Delta_f)_s. \]
\end{rem}
The affine Lie superalgebra of type $B^{(1)}(0,l)$ is, by definition, the Lie superalgebra $\fg(A,I_{odd})$ associated with the same GCM $A=(a_{i,j})_{0\leq i,j\leq l}$ as for type $BC_l^{(2)}$ introduced in \S \ref{subsect:root-data} and the subset $I_{\mathrm{odd}}:=\{l\}$ of the index set $I:=\{0,1,\cdots, l\}$.  
Thus, its Dynkin diagram is 
 \begin{center}
\begin{tikzpicture}
\draw (0,0) circle (0.1); \draw (0,-0.1) node[below] {$\alpha_0$};
\draw (0.1, 0) -- (0.9,0); 
\draw (0.45, 0.1) -- (0.55,0); 
\draw (0.45, -0.1) -- (0.55,0); 
\draw (0.5, 0.1) node[above] {$4$}; 
\fill (1,0) circle (0.1);  \draw (1,-0.1) node[below] {$\alpha_1$};
\draw (1.5,0) node[right] {$(l=1)$,};
\end{tikzpicture}
\quad
\begin{tikzpicture}
\draw (0,0) circle (0.1); \draw (0,-0.1) node[below] {$\alpha_0$};
\draw (0.45, 0.1) -- (0.55,0); 
\draw (0.45, -0.1) -- (0.55,0); 
\draw (0.5,0.1)  node[above] {$2$};
\draw (0.1,0) -- (0.9,0) ;
\draw (1,0) circle (0.1);  \draw (1,-0.1) node[below] {$\alpha_1$};
\draw (1.1,0) -- (1.9,0); 
\draw (2,0) circle (0.1); 
\draw [dashed] (2.1,0) -- (3.9,0);
\draw (4,0) circle (0.1); 
\draw (4.1,0) -- (4.9,0); 
\draw (5,0) circle (0.1); \draw (5,-0.1) node[below] {$\alpha_{l-1}$};
\draw (5.1,0) -- (5.9,0);
\fill (6,0) circle (0.1);  \draw (6,-0.1) node[below] {$\alpha_{l}$};
\draw (5.45,0.1) -- (5.55,0) ;
\draw (5.45,-0.1) -- (5.55,0) ;
\draw (5.5,0.1) node[above] {$2$};
\draw (6.5,0) node[right] {$(l\geq 2)$.};
\end{tikzpicture}
\end{center}
Here, the black nodes signify that the corresponding simple root $\alpha_l$ is non-isotropic and odd. 
The Weyl group $W$ is isomorphic to the Weyl group of type $BC_l^{(2)}$, and the set of real roots $\Delta^{re}$ is of type $BCC_l$,
i.e., $\Delta^{re}=\Delta_f+\Z \delta$, where $\delta:=\alpha_0+2(\alpha_1+\alpha_2+\cdots+\alpha_l)$ is the \textit{even} positive isotropic root such that $\Delta^{im}=\Z \delta$. The parity decomposition $\Delta=\Delta_{\bar{0}} \amalg \Delta_{\bar{1}}$ is given by
\[ \Delta_{\bar{0}}=\big((\Delta_f)_m+\Z \delta\big) \amalg \big((\Delta_f)_l+\Z \delta\big)\amalg \Delta^{im}, \qquad \Delta_{\bar{1}}=(\Delta_f)_s+\Z \delta. \] 
We remark that the node $\alpha_l$ is the only node verifying
$a_{l,j} \in 2\Z$ for any $0\leq j\leq l$. In particular, the map $\psi$ defined in \eqref{def_psi-map} is nice. 

\begin{rem} It is known  that the the derived subalgebra of $\fg(A, I_{odd})$ can be realized as the universal central extension of the loop algebra over  $\fg_f$. In particular, it follows that $\mult (n \delta)=l$ for any $n \in \Z \setminus \{0\}$. 
\end{rem}
For later use, let $\fn_{\pm}$ be the Lie sub superalgebras of $\fg(A, I_{odd})$ generated by $e_i$ ($i \in I$) and $f_i$ ($i \in I$), respectively. One can define the category $\cO$ of $\fg(A, I_{odd})$-modules as in the Kac-Moody case. Set $\Delta_+:=\Delta \cap \Z_{\geq 0}\Pi$.

\subsection{Integrability highest weight modules} 
A highest weight $\fg(A, I_{odd})$-module is defined similarly. 
For $\Lambda \in \fh^\ast$, let $\C_\Lambda=\C v_\Lambda$ be the one 
dimensional \textit{even} module over $\fb_+:=\fh \oplus \fn_+$ defined by 
\[ h.v_\Lambda=\langle h, \Lambda \rangle v_\lambda \quad (h \in \fh), \qquad 
\fn_+.v_\Lambda=0. \]
The induced $\fg(A,I_{odd})$-module  
$M(\Lambda):=\Ind_{\fb_+}^{\fg(A,I_{odd})}\C_\Lambda$ is called 
the Verma module with highest weight $\Lambda$. The Verma module $M(\Lambda)$ possesses the maximal proper submodule, say $N(\Lambda)$. 
Denote its irreducible quotient 
$M(\Lambda)/N(\Lambda)$ by $L(\Lambda)$. The next lemma can be shown, with the aide of Lemma \ref{lemma:IFDR-osp}:
\begin{lemma} The highest weight $\fg(A, I_{odd})$-module $L(\Lambda)$ is integrable if and only if $\Lambda(h_i) \in \Z_{\geq 0}$ for $0\leq i<l$ and $\Lambda(h_l) \in 2\Z_{\geq 0}$, i.e., $\Lambda \in P_+$ and the level of $\Lambda$ is even.
\end{lemma}
Note that the level of i-th fundamental weight  $\Lambda_i \in \fh^\ast$ is the co-label $a_i^\vee$. \\

As any highest weight module is $\fh$-diagonalisable, we can speak of its formal (super-)character which is an element of $\cE$. For an object $V$ of the cartegory $\cO$, let $V=\bigoplus_{\lambda \in \cP(V)} V_\lambda$ be its weight space decomposition. Its character $\ch(V)$ and super-character $\sch(V)$ are the elements of $\cE$ defined by
\[ \ch(V)=\sum_{\lambda \in \cP(V)} (\dim V_\lambda)e^{\lambda}, \qquad \sch(V)=\sum_{\lambda \in \cP(V)}(\sdim V_\lambda)e^\lambda, \]
where for a $\Z/2\Z$-graded vector space $E=E_{\bar{0}}\oplus E_{\bar{1}}$, we set 
\[\sdim(E)=\dim E_{\bar{0}}-\dim E_{\bar{1}}. \]

Let $\rho:=\sum_{i=0}^l \in \Lambda_i$ be a Weyl vector of $\fg(A, I_{odd})$. For a $\Lambda \in P_+$ with even level, 
the formal character of the irreducible highest weight module $L(\Lambda)$ is given by the Weyl-Kac character formula: 
\[ \ch L(\Lambda)=\sum_{w \in W}\vep(w)\ch M(w(\Lambda+\rho)-\rho)=\frac{\sum_{ w \in W}\vep(w)e^{w(\Lambda+\rho)}}{\sum_{w \in W}\vep(w)e^{w(\rho)}}. \]
This is exactly the same as for the affine Lie algebra $\fg(A)$ (cf. \eqref{Weyl-Kac}). 
Hence, the super-character of $L(\Lambda)$ can be given by
\[ \sch L( \Lambda)=\sum_{w \in W} \vep(w)\psi(w)\sch M(w(\Lambda+\rho)-\rho), \]
where the map $\psi: W \longrightarrow \{ \pm 1\}$ is defined in \eqref{def_psi-map}. By the Poincaré-Birkhoff-Witt theorem, it follows that 
\[ \sch M(\lambda)=e^{\lambda}\frac{\prod_{\alpha \in \Delta_+ \cap \Delta_{\bar{1}}}(1-e^{-\alpha})^{\mult(\alpha)}}{\prod_{\alpha \in \Delta_+ \cap \Delta_{\bar{0}}}(1-e^{-\alpha})^{\mult(\alpha)}} \]
which implies the super-denominator identity:
\begin{equation}\label{eqn:super-denominator-id}
 \sum_{w \in W} \vep(w)\psi(w)e^{w\rho}=e^{\rho}\frac{\prod_{\alpha \in \Delta_+ \cap \Delta_{\bar{0}}}(1-e^{-\alpha})^{\mult(\alpha)}}{\prod_{\alpha \in \Delta_+ \cap \Delta_{\bar{1}}}(1-e^{-\alpha})^{\mult(\alpha)}}. 
 \end{equation}
 As a corollary, one sees that
\[
\sch L(\Lambda)
=\frac{\sum_{ w \in W}\vep(w)\psi(w)e^{w(\Lambda+\rho)}}{\sum_{w \in W}\vep(w)\psi(w)e^{w(\rho)}}. \]
Thus, it is exactly the twisted-character \eqref{twisted-Weyl-Kac} !
\newpage

\bibliographystyle{plain}
\bibliography{Elliptic}

\begin{thebibliography}{10}

\bibitem{BruhatTits1972}
F.~Bruhat and J.~Tits.
\newblock Groupes r\'{e}ductifs sur un corps local.
\newblock {\em Inst. Hautes \'{E}tudes Sci. Publ. Math.}, (41):5--251, 1972.

\bibitem{Carter2005}
R.~W. Carter.
\newblock {\em Lie algebras of finite and affine type}, volume~96 of {\em
  Cambridge Studies in Advanced Mathematics}.
\newblock Cambridge University Press, Cambridge, 2005.

\bibitem{IoharaSaito202}
Kenji Iohara and Yoshihisa Saito.
\newblock Macdonald identities$:$ revisited.
\newblock {\em Jour. LIe Th.}

\bibitem{Kac1969}
V.~G. Kac.
\newblock Automorphisms of finite order of semisimple {L}ie algebras.
\newblock {\em Funk. Anal. i Prilo\v{z}.}, 3(3):94--96, 1969.

\bibitem{Kac1978}
V.~G. Kac.
\newblock Infinite-dimensional algebras, {D}edekind's {$\eta $}-function,
  classical {M}\"obius function and the very strange formula.
\newblock {\em Adv. in Math.}, 30(2):85--136, 1978.

\bibitem{Kac1990}
V.~G. Kac.
\newblock {\em Infinite-dimensional {L}ie algebras}.
\newblock Cambridge University Press, Cambridge, third edition, 1990.

\bibitem{KacPeterson1984}
V.~G. Kac and D.~H. Peterson.
\newblock Infinite-dimensional {L}ie algebras, theta functions and modular
  forms.
\newblock {\em Adv. in Math.}, 53(2):125--264, 1984.

\bibitem{Looijenga1976/77}
Eduard Looijenga.
\newblock Root systems and elliptic curves.
\newblock {\em Invent. Math.}, 38(1):17--32, 1976/77.

\bibitem{Macdonald1972}
I.~G. Macdonald.
\newblock Affine root systems and {D}edekind's {$\eta $}-function.
\newblock {\em Invent. Math.}, 15:91--143, 1972.

\bibitem{Moody1969}
R.~V. Moody.
\newblock Euclidean {L}ie algebras.
\newblock {\em Canad. J. Math.}, 21:1432--1454, 1969.

\bibitem{Saito1985}
K.~Saito.
\newblock Extended affine root systems. {I}. {C}oxeter transformations.
\newblock {\em Publ. Res. Inst. Math. Sci.}, 21(1):75--179, 1985.

\bibitem{Serre1970}
J.-P. Serre.
\newblock {\em Cours d'arithm\'etique}, volume~2 of {\em Collection SUP: ``Le
  Math\'ematicien''}.
\newblock Presses Universitaires de France, Paris, 1970.

\bibitem{Tits1979}
J.~Tits.
\newblock Reductive groups over local fields.
\newblock In {\em Automorphic forms, representations and {$L$}-functions
  ({P}roc. {S}ympos. {P}ure {M}ath., {O}regon {S}tate {U}niv., {C}orvallis,
  {O}re., 1977), {P}art 1}, Proc. Sympos. Pure Math., XXXIII, pages 29--69.
  Amer. Math. Soc., Providence, R.I., 1979.

\end{thebibliography}

\end{document}